\documentclass[11pt]{article}

\usepackage{amsthm}
\usepackage{amsmath}
\usepackage{amssymb}
\usepackage[margin=1in]{geometry}
\usepackage{enumerate}
\usepackage{hyperref}
\hypersetup{
  colorlinks=true, 
    linktoc=all,     
    linkcolor=blue,  
}
\usepackage{mathrsfs}
\usepackage{color}
\usepackage{bm}

\title{
Lyapunov exponents for random perturbations of 
coupled standard maps
}
\author{Alex Blumenthal, Jinxin Xue, Yun Yang}
\date{\today}

\newcommand{\Addresses}{{
  \bigskip
  \footnotesize

  Alex Blumenthal, \textsc{School of Mathematics, Georgia Institute of Technology,
  	Atlanta, Georgia, USA 30332}. 
	\textit{Email address}: \texttt{ablumenthal6@math.gatech.edu}
	
	\medskip
  
  Jinxin Xue, \textsc{Department of Mathematics, Yau Mathematical Sciences Center, Jingzhai 310, Tsinghua University, Beijing, China 100084}.
	\textit{Email address}: \texttt{jxue@tsinghua.edu.cn}
  
  \medskip

Yun Yang, \textsc{Department of Mathematics, Virginia Technical Institute,
Blacksburg, Virginia, USA 24061}. 
\textit{Email address}: \texttt{yunyang@vt.edu}

}}

\theoremstyle{theorem}
\newtheorem{thm}{Theorem}

\newtheorem{lem}[thm]{Lemma}
\newtheorem{prop}[thm]{Proposition}

\theoremstyle{definition}

\newtheorem{defn}[thm]{Definition}
\newtheorem{rmk}[thm]{Remark}

\newtheorem{cla}[thm]{Claim}

\newcommand{\E}{\mathbb{E}}

\newcommand{\N}{\mathbb{N}}
\renewcommand{\P}{\mathbb{P}}
\newcommand{\R}{\mathbb{R}}

\newcommand{\Z}{\mathbb{Z}}

\newcommand{\Uc}{\mathcal{U}}
\newcommand{\Bc}{\mathcal{B}}
\newcommand{\Fc}{\mathcal{F}}
\newcommand{\Hc}{\mathcal{H}}

\newcommand{\Gc}{\mathcal{G}}

\renewcommand{\Hc}{\mathcal{H}}
\newcommand{\Oc}{\mathcal{O}}

\newcommand{\Nc}{\mathcal{N}}

\newcommand{\Ec}{\mathcal E}

\renewcommand{\a}{\alpha}
\renewcommand{\b}{\beta}

\renewcommand{\d}{\delta}
\newcommand{\e}{\epsilon}

\newcommand{\pd}{\partial}

\newcommand{\graph}{\operatorname{graph}}

\newcommand{\Id}{\operatorname{Id}}

\newcommand{\T}{\mathbb T}
\newcommand{\Cc}{\mathcal C}
\newcommand{\Leb}{\operatorname{Leb}}

\renewcommand{\graph}{\operatorname{graph}}

\newcommand{\uo}{{\underline{\omega}}}

\newcommand{\Gr}{\operatorname{Gr}}
\newcommand{\Vc}{\mathcal V}

\newcommand{\Tr}{\operatorname{Tr}}

\newcommand{\modone}{\, (\text{mod }1)}

\author{Alex Blumenthal, Jinxin Xue, and Yun Yang}

\begin{document}
\maketitle

\begin{abstract}
In this paper, we give a quantitative estimate for the sum of the first $N$ Lyapunov exponents for random perturbations of a natural class $2N$-dimensional volume-preserving systems exhibiting strong hyperbolicity on a large but noninvariant subset of phase space. Concrete models covered by our setting include systems of coupled standard maps, in both `weak' and `strong' coupling regimes.
\end{abstract}
\tableofcontents
\section{Introduction}

Many systems, including large classes of those of physical interest,
exhibit strong sensitivity with respect to initial conditions. 
Mathematically, this behavior is described by 
 Lyapunov exponents:
for a smooth map $F : M \to M$ on a manifold $M$ and a point $x \in M$, 
the Lyapunov exponents of $F$ along the orbit $\{ F^i x\}$ are the possible values of
\[
\lambda(x, v) := \lim_{n \to \infty} \frac{1}{n} \log \| D_x F^n(v) \| \, , 
\]
when these limits exist, as $v$ ranges over tangent directions $T_x M$. 
If $\lambda(x, v) > 0$ for some $(x, v) \in T M$, then $F$ exhibits exponential
separation of trajectories in the phase space $M$ along the trajectory
$\{ F^i x\}$. For more discussion, see, e.g., \cite{barreira2002lyapunov, wilkinson2017lyapunov, young2013mathematical}.

%

Away from uniformly hyperbolic/Anosov settings, there can be extreme
challenges in actually verifying that a given system, 
even a simple, low-dimensional one,
admits positive Lyapunov exponents on an \emph{observable} subset of phase space (i.e., positive-volume). This can be the case even when a positive Lyapunov exponent is 
 ``obvious'' in numerical experiments. 

Exemplifying these challenges is the Chirikov standard map family
\[
(I, \theta) \mapsto \Phi_L (I, \theta) = (I + L \sin 2 \pi \theta, \theta + I + L \sin 2 \pi \theta) \, , 
\]
where $L \in \R$ is a real parameter and both coordinates $I, \theta$ are 
taken modulo $1$. 
Introduced by Chirikov \cite{Chirikov79},
the standard map is a fundamental toy model 
describing the dynamics along `stochastic boundary layers' formed by 
resonances in perturbed Hamiltonian systems, capturing the intricate interaction in phase space between `regular' (elliptic) and `chaotic' (seemingly stochastic) motion. For more discussion,
see \cite{chirikov2008chirikov}.

When $L \gg 1$, the mapping $\Phi_L$ exhibits strong hyperbolicity, 
except along an $O(L^{-1})$ neighborhood of the vertical lines 
$\{ \theta = \pi/2, 3 \pi / 2\}$. The volume of these 
\emph{critical strips}, where hyperbolicity fails, approaches 0 as $L \to \infty$, and so
one might expect $\lambda(x, v) > 0$ for most $x$; this is corroborated by a wealth of 
numerical evidence. However, to prove this mathematically rigorously is
a notorious open problem: to date, no-one has proved that $\Phi_L$ admits a positive Lyapunov
exponent on a positive-area set for any value of $L$ (equivalently, by Pesin's entropy formula, that
Lebesgue measure has positive metric entropy for $\Phi_L$). 
The primary challenge to overcome is \emph{cone-twisting}, i.e., when previously expanded
tangent directions can, upon the trajectory entering the critical set, be `twisted'
into strongly contracting directions. 
Estimating Lyapunov exponents for models of this kind
amounts to an incredibly delicate cancellation problem between phases of growth
and decay, all depending on the (time-varying) orientation of tangent directions. 

These challenges are real, 
as evidenced by known results on the relative density of elliptic periodic orbits 
in phase space (e.g., \cite{duarte1994plenty}), which imply that even when $L$ is taken arbitrarily
large, there may be positive-area regions of phase space with zero Lyapunov exponents. 
In the positive direction, Gorodetski has shown that for a residual subset of $[L_0, \infty)$, $L_0 \gg 1$ taken sufficiently large, the set with a positive Lyapunov exponent has Hausdorff dimension 2  \cite{gorodetski2012stochastic}, although this is quite far from a positive-area set. 
We also mention the more recent work of Berger and Turaev, whose work on perturbations of elliptic islands for surface diffeomorphisms implies that $\Phi_L$
is $C^\infty$ close to a volume-preserving mapping admitting a positive Lyapunov exponent on a 
positive-area set \cite{Berge19}. 
Lastly, we note that this brief discussion omits many works and indeed entire subfields related to the 
standard map, e.g., Schr\"odinger cocycles. 
We refer the readers to the introduction of \cite{BXY1} for more discussion.

\subsubsection*{Tractability of Lyapunov exponents in the presence of noise}

The problem of estimating or computing Lyapunov exponents is far more tractable when the dynamics is subjected to sufficiently ``nondegenerate'' (i.e., absolutely continuous) IID random perturbations applied at each timestep.

To establish these ideas, let us summarize the results from the work
\cite{BXY1} on Lyapunov exponents for 
random perturbations of the standard map $\Phi_L$. 
Let $\uo := (\omega_1, \omega_2, \cdots)$ be an IID sequence of 
random variables distributed uniformly on $[-\e, \e]$, where $\e > 0$ is a small parameter.
For $n \geq 1$, consider the random compositions 
\[
F^n_\uo = F_{\omega_n} \circ \cdots \circ F_{\omega_1} \, , \quad F_\omega(\theta, I) := F(\theta + \omega, I) \,. 
\]
It is not hard to show (Lemma 5 of \cite{BXY1}) that for any $\e > 0, \exists \lambda_1^\e \geq 0$ such that for any fixed $x \in T_x \T^2$, we have
$\lambda_1^\e = \lim_{n \to \infty} \frac{1}{n} \log \| D_x F^n_\uo\|$ with probability 1. 
Of interest is whether or not 
\begin{align}\label{eq:reflects}
\lambda_1^\e \approx \log L \, , 
\end{align} 
i.e., whether the top Lyapunov exponent $\lambda_1^\e$ actually `reflects' the fact that $F$ has such strong hyperbolicity on 
a large proportion of phase space.\footnote{It is a folklore theorem \cite{Mattingly} that $\lambda_1^\e > 0$ 
for all $\e > 0$, using a generalization of Furstenberg's seminal work \cite{furstenberg1963noncommuting} on Lyapunov 
exponents of IID matrices (e.g., \cite{carverhill1987furstenberg} or \cite{ledrappier1986positivity}). However, we note that this method yields no quantitative information on $\lambda_1^\e$.}


For the standard map $F = \Phi_L$, it was proved in \cite{BXY1} that 
for large $L$ and for all $\e \geq e^{- L}$, equation \eqref{eq:reflects} holds. 
Of particular interest is the fact that this result allows the deterministic map $F$
to have elliptic islands: when $\e \approx e^{-L}$, trajectories of the corresponding
random maps can linger in islands for very long timescales. This is surprising, 
in view of the fact that cone twisting near elliptic islands is a major source of difficulty
in studies of the (deterministic) standard map. 

In comparison with approaches to Lyapunov exponents for deterministic models
exhibiting cone twisting (for example, work on the Henon map, e.g., \cite{benedicks1991Henon}, and work on rank-one attractors, e.g., \cite{wang2001strange}), the method used in \cite{BXY1} is quite short and straightforward. 
Results in a similar vein include the works 
\cite{lian2012positive, blumenthal2018positive} (see also \cite{galatolo2017existence}), which study random perturbations of 
1D mappings with a critical point, and the works \cite{BXY2, ledrappier2003random} studying random perturbations of 2D mappings exhibiting cone twisting. 

\subsubsection*{Results in this paper}

The aim of this paper is to extend this program to a class of volume-preserving systems of arbitrarily high dimension exhibiting strong hyperbolicity on a large yet noninvariant subset of phase space. 
We aim to make estimates on all Lyapunov exponents, not just the `top' exponent. 
Although our approach in this paper is inspired by that of \cite{BXY1}, the 
higher-dimensional setting introduces several new layers of complexity which must be 
contended with; see Section \ref{subsec:commentsAndComparison} for more discussion. 

The abstract setting we introduce below includes systems of coupled
Chirikov standard maps in a variety of coupling regimes: 
for $N > 1$ we consider $N$
standard map oscillators $(I_i, \theta_i) \in \T^1 \times \T^1, i = 1,\cdots, N$, with a time evolution $(I_i, \theta_i)
\mapsto (\bar I_i, \bar \theta_i)$ defined by
\begin{align}\label{eq:standardMapFormula}
\begin{aligned}
 \overline{I}_i & = I_i + L \sin 2 \pi \theta_i + \sum_{j \neq i} \mu_{ij} \sin 2 \pi (\theta_j - \theta_i) \, ,\\
\overline{\theta}_i & = \theta_i + \bar I_i  \, ,
\end{aligned}
\end{align} 
where $\T^1$ is parametrized as $[0,1)$, with all quantities above regarded ``modulo 1''. 
This is a completely integrable, uncoupled system when $L, (\mu_{ij})$ are zero; 
in this paper, we will instead be interested in the so-called \emph{anti-integrable}
regime where $L \gg 1$ and the $(\mu_{ij})$ can be potentially quite large. 
Note that the 
above mapping is symplectic, hence volume-preserving, iff $\mu_{ij} = \mu_{ji}$ for all $i,j$. 

Coupled standard maps appear in the physical literature as toy models of Arnold diffusion \cite{Chirikov79} as well as the statistical properties of chaotic maps. These maps exhibit strong evidence of chaotic behavior in experiments, while mathematically rigorous verification of this chaotic behavior is hopelessly out of reach in the absence of noise. 
 We refer the readers to, e.g., 
 \cite{Guido03, Holger88, Manos08,Blake90} and the references therein for more physics background and research on coupled standard maps.

\subsection{Deterministic maps from which we perturb}

Let $\T = \T^1$ denote the circle, parametrized as $[0,1) \cong \R / \Z$, equipped with the usual addition interpreted `modulo 1'. For $N \geq 1$ we consider dynamics on the torus 
 $\T^{2N} \cong \R^{2N} / \Z^{2N}$, which we regard with the flat metric coming from $\R^{2N}$. 
Throughout we identify $T \T^{2N} \cong \T^{2N} \times \R^{2N}$.

For $N \geq 1$ we consider one-parameter families $F_L : \T^{2N} \to \T^{2N}$ 
of the form
\begin{align}\label{eq:formOfF}
 F_L(x,y) = (f_L(x) - y, x)
\end{align}
where $x = (x_1, \cdots, x_N) \in \T^N, y = (y_1, \cdots, y_n) \in \T^N$ and $ f_L : \T^{N} \to \R^N, L \geq 1$ is a one-parameter family of smooth maps. The expression $f_L(x) - y$ is interpreted `modulo 1' in all coordinates. Note that $F_L$ is invertible and volume-preserving.

We assume throughout that $f_L$ satisfies the following:
\begin{itemize}
\item[(F1)] There exists $C_0 > 0$ such that $\| D_x f_L\| \leq C_0 L$ for all $L$; and
\item[(F2)] For any $\beta \in (0,1)$, there exist $C_\beta,  c_\beta, L_\beta > 0$ so that 
for all $L \geq L_\beta$, we have that
\[
B_\beta = \{ x \in \T^N :  |\det  D_x f_L | \leq L^{N - (1 - \beta)} \} \subset \T^N
\]
has Lebesgue measure $\leq C_\beta L^{- c_\beta}$.
\end{itemize}
Families $(F_L)$ satisfying these conditions
 include a wide array of systems, including 
systems of identical coupled standard maps: 
It is straightforward to show that the system \eqref{eq:standardMapFormula}
is conjugate to $F_L$, with 
\begin{align}\label{eq:formfLstandard}
f_L(x) = \big(2 x_i + L \sin 2 \pi x_i + \sum_{j \neq i} \mu_{ij} \sin 2 \pi (x_j - x_i) \big)_{i = 1}^N \, ,
\end{align}
via the change of coordinates 
$x_i = \theta_i, \quad y_i = \theta_i - I_i \modone \, .
$

One should keep in mind the example
\begin{align}\label{eq:fLform}
f_L = L \psi + \varphi \, , 
\end{align}
where $\psi, \varphi : \T^N \to \R^N$ are fixed and $L$ is taken large. 
For $f_L$ of this form, condition (F1) is evident, 
while condition (F2) holds with $c_\beta = 1- \beta$ when
$\psi$ satisfies the transversality-type condition
\begin{align}\label{eq:transConditionIntro}
\{ \det D_x \psi = 0 \} \cap \{ \nabla \det D_x \psi = 0 \} = \emptyset \, ;
\end{align}
see Lemma \ref{lem:F2Generic}.  For $N = 2$, we show (Proposition \ref{prop:F2genericN2}) that
\eqref{eq:transConditionIntro} holds for a $C^2$-generic set of $\psi$.

\subsubsection*{`Predominant' hyperbolicity of $F_L$}

Let us describe briefly the hyperbolic character of the dynamics of $F_L$. 
Throughout, we identity $T \T^{2N} \cong \T^{2N} \times \R^{2N}$. 
We write $\R^{2N} = \R^x \oplus \R^y$, where $\R^x = \operatorname{Span} \{ \frac{\pd}{\pd x_1}, \cdots, \frac{\pd}{\pd x_N}\}$ and $\R^y = \operatorname{Span}\{ \frac{\pd}{\pd y_1}, \cdots, \frac{\pd}{\pd y_N}\}$, each of which is parametrized by $\R^N$. For $\alpha > 0$, we define
\[
\Cc^x_\a := \{ (u,v) \in \R^{2N} : \| v \| \leq \a \| u \| \}
\]
of vectors $\alpha$-close to the `horizontal' space $\R^x$. 

Conditions (F1) and (F2) ensure uniform expansion to order $L$ of tangent vectors
in $\Cc^x_\alpha$ for moderate values of $\alpha$. To wit, if $z = (x,y)$ is such that
$\det(D_x f_L) \geq L^{N - (1 - \beta)}$, then (F1) implies\footnote{Here and throughout the paper, if $A = A(L), B = B(L)$ are functions of $L$, 
we write $A \lesssim B$ if $\exists C > 0$, independent of $L$, such that
$A(L) \leq C B(L)$ for all $L$ sufficiently large.}
\begin{align}\label{eq:predomHyp}
D_z F_L(\Cc^x_{1/10}) \subset \Cc^x_{1/10} \,  \quad \text{ and } \quad 
\inf_{\substack{w \in \Cc^x_{1/10} \\ \| w \| = 1}} \| D_z F_L(w) \| \gtrsim L^\beta. 
\end{align}
See Lemma \ref{lem:bulkHyp} for more details.
The `critical' set where this hyperbolicity fails is contained
in $\det(D_x f_L) \leq L^{N - (1 - \beta)}$, which by (F2) has volume $\lesssim L^{- c_\beta}$. 
We call $F_L$ \emph{predominantly hyperbolic}, since for $L \gg 1$ the strong expansion in 
equation \eqref{eq:predomHyp} holds on a large (but noninvariant) proportion of phase space.

\subsection{{Random dynamical systems (RDS)} setup}\label{subsec:noiseModel}

Fix a probability space $(\Omega_0, \Fc_0, \P_0)$ and let $\omega \mapsto R_\omega \in C^2_{\rm vol}(\T^{2N}, \T^{2N})$ be a measurable assignment to each $\omega \in \Omega_0$ of a 
$C^2$, volume-preserving diffeomorphism $R_\omega : \T^{2N} \to \T^{2N}$, to be interpreted
as the `noise' applied to the dynamics at each timestep.
Define $\Omega = \Omega_0^{\otimes \N}, \Fc = \Fc_0^{\otimes N}, \P = \P_0^{\otimes N}$
and let $\theta : \Omega \to \Omega$ be the leftward shift (which is automatically invariant and 
ergodic for $\P$). Elements $\uo \in \Omega$ are written $\uo = (\omega_1, \omega_2, \cdots)$ for
$\omega_i \in \Omega_0, i \geq 1$. 

In this paper, we consider random compositions 
\[
F^n_\uo = F_{\omega_n} \circ \cdots \circ F_{\omega_1} \, , \quad n \geq 1, \quad \uo = (\omega_i)_{i \in \N} \in \Omega
\]
of the (IID) random maps
\[
F_\omega = R_\omega \circ F \, , \quad \omega \in \Omega_0 \, ,
\]
where $F = F_L$ is as in \eqref{eq:formOfF}. 

Throughout, we assume the following properties (E), (C) and (ND) of the $R_\omega$, where E, C and ND are abbreviations for \emph{Ergodic, Cone} and \emph{Nondegeneracy}, respectively.

\begin{itemize}
\item[(E)] For $z \in \T^{2N}$, the law $Q(z, \cdot) = \P_0(R_\omega z \in \cdot)$ on $\T^{2N}$ 
is absolutely continuous. Writing $q(z, \cdot) = \frac{d Q(z, \cdot)}{d \Leb}: \T^{2N} \to \R_{\geq 0}$ for the corresponding density, we assume $\exists c > 0$ such that
\begin{align}
q(z, z') > 0 \quad \text{ for all } \, z' \in B_c(z) \, , 
\end{align}
where $B_c(z)$ is the open $c$-ball centered at $z \in \T^{2N}$. 
\item[(C)] With probability $1$ and for all $z \in \T^{2N}$, we have:
\begin{itemize}
\item $D_z R_\omega (\Cc^x_{1/20}) \subset \Cc^x_{1/10}$ ; and
\item $\| D_z R_\omega\|, \| (D_z R_\omega)^{-1} \| \leq 2$.
\end{itemize}
\item[(ND)] For $(z, E) \in \Gr_N(\T^{2N})$, the measure 
\begin{align}\label{eq:defnhatQ}
\hat Q((z, E), \cdot) = \P_0 ( (R_\omega z, D_z R_\omega(E)) \in \cdot)
\end{align}
on $\Gr_N(\T^{2N})$ is absolutely continuous with respect to the (normalized) Riemannian volume $\mathfrak m$
 on $\Gr_N(\T^{2N})$. Its density $\hat q((z, E), \cdot) := \frac{d \hat Q((z, E), \cdot)}{d \mathfrak m}$
satisfies
\[
M := \sup_{(z, E) \in \Gr_N(\T^{2N})} \| \hat q((z, E), \cdot) \|_{L^\infty} < \infty \, .
\] 
\end{itemize}

Above, for $1 \leq k < m$, we write $\Gr_k(\R^m)$ for the Grassmanian of $k$-dimensional subspaces of $\R^m$. We write $\Gr_k(\T^{2N}) \cong \T^{2N} \times \Gr_k(\R^{2N})$ for the 
Grassmanian bundle of $k$-planes in tangent space $T \T^{2N}$, and 
\[
\mathfrak m = \Leb_{\T^{2N}} \times \Leb_{\Gr_N(\R^{2N})}
\]
for the (normalized) Riemannian volume on the 
Grassmanian bundle $\Gr_N(\T^{2N}) \cong \T^{2N} \times \Gr_N(\R^{2N})$.

Condition (C) ensures that the randomizations we use do not introduce more cone-twisting
than already present, while condition (E) ensures almost-sure constancy of Lyapunov exponents: 
Lyapunov exponents of the random compositions
$F^n_\uo$ exist a.s. and with probability 1 by the Multiplicative Ergodic Theorem.
Precisely, we have the following:

\begin{lem}\label{lem:existLE}
Assume condition $(E)$ holds. Then, there exist $2N$ (deterministic) real constants 
$\lambda_1 \geq \cdots \geq \lambda_{2N}$ such that
for $\P \times \Leb_{\T^{2N}}$-a.e. $(\uo, z) \in \Omega \times \T^{2N}$ and any $v \in T_z \T^{2N}$, we have that
\[
\lim_{n \to \infty} \frac{1}{n} \log \| D_z F^n_\uo(v) \|
\]
exists and equals $\lambda_i$ 
for some $i$. The \emph{Lyapunov exponents} $(\lambda_i)$ satisfy\footnote{For a matrix $A$ we write $\sigma_1(A), \sigma_2(A), \cdots$ for the singular values of $A$. See Appendix
\ref{sec:SVDapp} for more discussion. }
\begin{align}
\lambda_i = \lim_{n \to \infty} \frac{1}{n} \log \sigma_i(D_z F^n_\uo) 
\end{align}
for $\P \times \Leb_{\T^{2N}}$-a.e. $(\uo, z)$.

\end{lem}

For the proof of Lemma \ref{lem:existLE}, see Section \ref{subsec:relevMarkovGrass}. 
 For commentary on the role of condition (ND), 
see Section \ref{subsec:commentsAndComparison} below. 
Explicit examples of noise models $R_\omega$ satisfying (ND) are
 constructed in Section \ref{sec:noiseModel}.

\begin{rmk}
 Consider a family $R_\omega^\epsilon, \e > 0$ of random perturbations satisfying (ND)
 for which $\lim_{\e \to 0} d_{C^2}(R_\omega^\e, \Id) = 0$, where $\Id$ is the identity
mapping on $\T^{2N}$. The corresponding
bounds $M^\epsilon = \sup \| \hat q^\e((z, E), \cdot)\|_{L^\infty}$ in condition (ND) would satisfy $M^\e \to \infty$ as $\e\rightarrow 0$. Thus, 
$M$ measures, in a statistical sense, how close the $R_\omega$ are to
the identity mapping: when $R_\omega$ is $\P_0$-typically very close to the identity mapping, 
$M$ is very large (note that the converse is false: $M$ may be large even when $d_{C^2}(R_\omega, \Id)$ is typically of order 1).  
\end{rmk}

\subsection{Results}\label{subsection:results}
Our main results estimate all Lyapunov exponents $(\lambda_i)$.

\begin{thm}\label{thm:main}
Assume the family $f = f_L$ satisfies $($F1$)$, $($F2$)$ above and the randomizations $R_\omega$ satisfy $(E)$, $(C)$ and $($ND$)$. 
Fix $\alpha \in (0,1), \beta \in (0,1)$ and $\delta\in (0,c_\beta)$ where $c_\beta$ is in (F2). Let $L$ be sufficiently large in terms of these parameters. Finally, assume $
M \leq L^{\frac12 \beta L^{c_\beta - \delta} }$
where $M$ is as in condition $($ND$)$ above. Then,  
the Lyapunov exponents $(\lambda_i)$ of $F^n_\uo = F_{\omega_n} \circ \cdots \circ F_{\omega_1}, F_\omega := 
R_\omega \circ F$, satisfy
\begin{align}\label{eq:indivLyapEst}
\lambda_N > 0 > \lambda_{N + 1} \, , \quad \text{ and } \quad  \min\{ | \lambda_i | \} \geq \alpha \log L \, .
\end{align}
\end{thm}

\begin{rmk}
Observe that (ND) is imposed only on the randomization $R_\omega$ independently of the 
 deterministic dynamics $F_L$. Thus, 
  upper bounds of the form $M\leq G(L)$, $G(L)$ an increasing function in $L$,
 are easily satisfied by taking $L$ large. 
\end{rmk}

Theorem \ref{thm:main} applies to a wide class of identical coupled standard maps with
relatively general forms of coupling. 
\begin{thm} \label{ThmStd}
Let $\alpha, \beta \in (0,1)$ and $\delta > 0$ {be such that $1-\beta-\delta>0$}. {Let $N \geq2$ and consider a family of coupled standard maps $F_L$ as in
 \eqref{eq:formOfF} with 
 \[
f_L(x) = \big(2 x_i + L \sin 2 \pi x_i + \sum_{j \neq i} \mu_{ij} \sin 2 \pi (x_j - x_i) \big)_{i = 1}^N \, ,
\]
where the coefficients $\mu_{ij}$ are fixed and $L$ is sufficiently large depending on $(\mu_{ij})$.  
Assume that the randomizations $R_\omega$ satisfy (E), (C) and (ND), with
$
M \leq L^{\frac16 \beta L^{1 - \beta - \delta}}
$
where $M$ is as in condition $($ND$)$ above. 
Then, 
the Lyapunov exponents $(\lambda_i)$ of $F^n_\uo = F_{\omega_n} \circ \cdots \circ F_{\omega_1}, F_\omega := 
R_\omega \circ F$, satisfy
$$
\lambda_N > 0 > \lambda_{N + 1} \, , \quad \text{ and } \quad  \min\{ | \lambda_i | \} \geq \alpha \log L \,.$$}
\end{thm}

The setting of Theorem \ref{ThmStd} 
can be thought of as describing a kind of `weak' to `moderate' coupling 
regime: the strength of the hyperbolicity $L$ of each individual oscillator overshadows the 
coupling amplitude $\max_{ij} | \mu_{ij}|$. 
The following applies in a regime when the strength of the coupling
matches that of the individual oscillators. 

\begin{thm}\label{thm:strongCoupling}
Let $\alpha, \beta \in (0,1)$ and $\delta > 0$ be such that $1-\beta-\delta>0$. 
Let $N = 2$ and consider a family of coupled standard maps $F = F_L$ as in \eqref{eq:formOfF} with
\[
f_L(x_1, x_2) = \begin{pmatrix}
2 x_1 + L \sin 2 \pi x_1 + L \sin 2 \pi (x_2 - x_1) \\
2 x_2 + L \sin 2 \pi x_2 + L \sin 2 \pi (x_1 - x_2) 
\end{pmatrix} \, .
\]
where  $L$ is sufficiently large.  
Assume that the randomizations $R_\omega$ satisfy $(E)$, $(C)$ and $($ND$)$, with 
$M  \leq L^{\frac12 \beta L^{1 - \beta - \delta}}$
where $M$ is as in condition $($ND$)$ above. 
Then, 
the Lyapunov exponents $(\lambda_i)$ of $F^n_\uo = F_{\omega_n} \circ \cdots \circ F_{\omega_1}, F_\omega := 
R_\omega \circ F$, satisfy
$$
\lambda_N > 0 > \lambda_{N + 1}\ \mathrm{and} \ \min\{ | \lambda_i | \} \geq \alpha \log L \,.$$
\end{thm}

\begin{rmk}
The bulk of the work in applying Theorem \ref{thm:main} to the coupled standard
maps in Theorems \ref{ThmStd}, \ref{thm:strongCoupling} is to verify the (F2) condition.
 Although Theorem \ref{thm:strongCoupling} has only been checked in the case
$N = 2$, the conditions used to derive it should be checkable for any reasonable
value of $N$. See Section \ref{sec:app} for more discussion. 
\end{rmk}
\begin{rmk}

We emphasize the order of quantifiers in our results: throughout, we fix the dimension $N$ and choose $L$ sufficiently large depending on $N$, but not vice versa. It would, though, be of interest to fix $L$, take the `hydrodynamic limit' $N \to \infty$ and study the resulting Lyapunov spectrum. Limits of this kind provide toy models for, e.g., gases of particles in the hydrodynamic limit; 
see, e.g., \cite{yang2009lyapunov, yang2006dynamical}. 
However, this is beyond the scope of the present paper, since our analysis here does not take into account quantitative dependence in $N$.
\end{rmk}

\subsection{Comments and comparison with prior work}\label{subsec:commentsAndComparison}

This paper is inspired by the approach in \cite{BXY1}, which studied Lyapunov exponents 
for the random compositions $F^n_\uo = F_{\omega_n} \circ \cdots \circ F_{\omega_1} : \T^2 \circlearrowleft$, where
\begin{gather*}
F(x,y) = (2 x + L \sin(2 \pi x) - y, x) \, , \quad F_\omega(x,y) = F\circ R_\omega(x,y) \, ,\\
R_\omega(x,y) := (x  + \omega, y) \, , 
\end{gather*}
and the random perturbations $\omega_i$ are IID uniformly distributed in $[- \e, \e]$, where
$\e > 0$ is a small parameter. 
Lyapunov exponents are estimated by considering the Markov chain $(Z_n, V_n) \in \T^2 \times S^1$,
\begin{align}\label{zvprocess}
Z_n = F^n_\uo(Z_0) \, , \quad V_n = \frac{D_{Z_0} F^n_\uo(V_0)}{\| D_{Z_0} F^n_\uo(V_0) \|} \, ,
\end{align}
using the well-known fact that stationary measures $\nu^\e$ for $(Z_n, V_n)$ are related to Lyapunov exponents by the formula 
\[
\lambda_1 \geq \E \int_{\T^2 \times S^1} \log \| D_z F_\omega(v)\| d \nu^\e(z, v) \, ,
\]
where $\lambda_1$ is the top Lyapunov exponent for $(F^n_\uo)$ (see Kifer \cite{kifer1982perturbations}).

What is shown in \cite{BXY1} is that $\nu^\e$ mass is largely concentrated
away from contracting directions roughly parallel to $\frac{\pd}{\pd y}$, resulting in a lower
bound $\lambda_1 \gtrsim \log L$. 
Estimates on $\nu^\e$ itself are derived
by combining: 
\begin{itemize}
\item[(i)] strong hyperbolic expansion on a large (but noninvariant) 
subset of phase space with 
\item[(ii)] a priori upper bounds on the 
transition kernel for $(Z_n, V_n)$, which result in a priori estimates for the density of $\nu^\e$ (Lemma 9 in \cite{BXY1}).
\end{itemize}

The present paper applies these ideas to the Markov chain $(Z_n, E_n)$ keeping track of a base point $Z_n \in \T^{2N}$ and an $N$-dimensional subspace $E_n \subset T_{Z_n} \T^{2N}$ 
of tangent directions. Extending item (i) above, 
the natural separation of the degrees of freedom of $\T^{2N}$ into the `expanding' $(x_1, \cdots, x_N)$ and `contracting' $(y_1, \cdots, y_N)$ ensures strong hyperbolic expansion along $N$-dimensional subspaces roughly parallel
to $\operatorname{Span}\{ \frac{\pd}{\pd x_1}, \cdots, \frac{\pd}{\pd x_N}\}$. 
Generalizing item (ii) above, control on the transition kernel for $(Z_n, E_n)$ is provided by 
assumption (ND), which immediately implies non-concentration 
for stationary measures for $(Z_n, E_n)$ (see Lemma \ref{lem:aPrioriBoundND}). 
 In particular, the assumption precludes the concentration of stationary mass in a neighborhood of the contracting direction $\R^y$, which is the main obstruction to having positive or large Lyapunov exponents. 
 
 Our higher dimensional setting entails several challenges, the most obvious of which is that the relevant tangent space dynamics is much more complicated. To illustrate this point, note that 
 the Markov chain $(Z_n, V_n)$ considered in \cite{BXY1} lives on the 3D space $\T^2 \times S^1$, while in our setting $(Z_n, E_n)$ lives on the $(2N + N^2)$-dimensional space $\T^{2N} \times \Gr_N(\R^{2N})$. 
 In \cite{BXY1}, hyperbolicity as in (i) above means that $V_n$ is repelled from 
 $\R^y$ and drawn towards $\R^x$ as long as $(Z_n = (X_n, Y_n))$ avoids the critical strips $\{ x = \frac14, \frac34\}$. This picture is much more complicated in our setting: 
although $E_n$ is still repelled from $\R^y$ and attracted by $\R^x$, it also experiences 
the influence of a vast heteroclinic network of saddle-type behavior 
near `hybrid' subspaces of the form
$\operatorname{Span}\{ \frac{\pd}{\pd x_{i_1}}, \cdots, \frac{\pd}{\pd x_{i_l}}, \frac{\pd}{\pd y_{j_1}}, \cdots, \frac{\pd}{\pd y_{j_{N-l}}} \}$, where $1 \leq i_1 < \cdots < i_l \leq N, 1 \leq j_1 < \cdots < j_{N-l} \leq N$ are arbitrary indices. 

\subsubsection*{Comments on nondegeneracy assumption (ND)}

For $N = 1$, note that the randomization $R_\omega(x,y) = (x + \omega, y)$ forces only the $x$-coordinate, and so condition (ND) cannot possibly hold for this noise model. 
From the perspective of the $(Z_n, V_n)$ Markov chain as in \eqref{zvprocess}, 
this noise is `degenerate', and
only propagates to noisy forcing of all three degrees of freedom of $(Z_n, V_n)$ after three iterates.
This is made formal in Lemma 9 of \cite{BXY1}, which 
 bounds the density of the time-3 transition kernel
\[
\hat P^3((z, v) , \cdot) = \P\left( \big( F^3_\uo(z), \frac{D_z F^3_\uo(v)}{\|D_z F^3_\uo(v)\|} \big) \in \cdot \right) \, .
\] 
In contrast, (ND) 
ensures that the noise $R_\omega$ acts `nondegenerately' on $(Z_n, E_n)$. 
Drawing an analogy with continuous-time stochastic differential equations, 
forcing of the type $R_\omega(x,y) = (x + \omega, y)$ is `hypoelliptic', whereas
 (ND) ensures that $R_\omega$ provides `elliptic'-type forcing for $(Z_n, E_n)$.

\begin{rmk}
It is probably possible to extend the analysis in this manuscript to 
`hypoelliptic'-type noise of the form $R_\omega(x,y) = (x + \omega, y)$.
However, the dimension of the Grassmanian manifold of $N$-dimensional subspaces
$\Gr_N(\R^{2N})$ of $\R^{2N}$ grows like $N^2$ as $N$ gets large, which is much larger than the dimension of the base manifold. Thus, `propagating' noise from the base dynamics to the entire Grassmannian bundle involves estimating time-$N$ transition kernels, which is 
computationally quite involved. We leave this problem for future work. 
\end{rmk}


\subsubsection*{Additional related prior work}

Related to our setting is the work of Berger and Carrasco \cite{Berge14}, who considered the Lyapunov exponents of a skew product of a hyperbolic CAT map with a Chirikov standard map. This was generalized recently by Carrasco \cite{Carrasco19} to estimate the Lyapunov exponents of arbitrarily many coupled standard maps. 
Applying a symbolic coding to the CAT map, one can view the models in \cite{Berge14, Carrasco19} as random perturbations by discrete noise (by comparison, \cite{BXY1} and this paper both use the absolutely continuous noise). 

However, we emphasize that both the models considered and the techniques used in 
\cite{Berge14, Carrasco19} are highly different from our setting. 
The most significant difference is that the perturbations applied to the standard map
 in \cite{Berge14, Carrasco19} are necessarily of order 1, and so the perturbed and unperturbed mappings have completely different dynamics even after 1 timestep. In contrast, 
the perturbations in  \cite{BXY1} and this paper may be extremely small, so that the perturbed
and unperturbed mappings are close even after many timesteps. This ``largeness'' of the perturbation in \cite{Berge14, Carrasco19} is inherent to the methods used: 
the skew products considered are set up 
so that strong expansion from the CAT map dominates the dynamics of the standard map. Thus, their model admits a global foliation by strongly-unstable manifolds 
which cross the entire domain of the standard map in a uniform way- this strong
geometric property is their primary tool for estimating Lyapunov exponents. 

Farther from our work, there is a wealth of literature on Lyapunov exponents. 
We mention, for instance, Furstenberg's famous
1963 paper \cite{furstenberg1963noncommuting} on positivity of Lyapunov exponents for IID products of determinant 1 matrices, 
and the vigorous activity that followed extending this work to random products of matrices
driven by more general processes (e.g., \cite{virtser1980products, guivarc50products}) and to simplicity of the Lyapunov spectrum (e.g., \cite{gol1989lyapunov}). We emphasize, though, that these works are qualitative and a priori provide no concrete estimates of Lyapunov exponents. 

We have only emphasized here works which directly address nonuniform hyperbolicity (in the presence of cone twisting) only in high-dimensional systems. For a broader discussion, we refer the reader to the introduction of \cite{BXY1}. 

\subsubsection*{Organization of the paper}

In Section \ref{SPrelim}, we give some preliminary results on the Markov chain on $\T^{2N}$ and on $\Gr_N(\T^{2N})$, while in Sections \ref{SThmMain} and \ref{SProp} we prove
Theorem \ref{thm:main}. Theorems \ref{ThmStd} and \ref{thm:strongCoupling} are
proved in Section \ref{sec:app}. 
Sufficient conditions for (F2) and genericity results, as well as applications to coupled
standard maps, are worked out in Section \ref{sec:app}. 
In Section \ref{sec:noiseModel} we construct an explicit
example of a noise model $R_\omega$ satisfying conditions (E), (C) and (ND).
Included in Appendix \ref{sec:SVDapp} is a version of the standard 
singular value decomposition used in this paper.

\section{Preliminaries}\label{SPrelim}
\subsection{The Grassmanian as a Riemannian manifold}

 Fix $m \geq 1$ and $1 \leq k < m$. Here we describe the smooth and Riemannian structures 
 of the manifold $\Gr_k(\R^m)$ of 
$k$-dimenisonal subspaces of $\R^m$, and give a few preliminary lemmas. 
The following is all well-known; see, e.g., \cite{nicolaescu2007lectures, piccione2000geometry}.

To fix ideas and avoid 
dealing with unnecessary cases, we will exclusively deal with the case when $k \leq \frac{m}{2}$, 
hence $k \leq m - k$. Otherwise, we can reduce to this case by noting that orthogonal projection
provides a natural identification $\Gr_k(\R^m) \cong \Gr_{m-k} ( \R^m)$. 
Throughout, $\R^m$ carries the standard Euclidean inner product
$\langle \cdot ,\cdot \rangle$.

Given $E \in \Gr_k(\R^m)$, we define the coordinate patch $\Uc_E = \{ \graph_E H : H \in L(E, E^\perp)\}$,
where we write $L(E, E^{\perp})$ for the space of linear maps from $E$ to $E^{\perp}$, and
the chart map
\[
\graph_E : \Uc_E \to \Gr_k(\R^m)
\]
is defined by $\graph_E H = \{ v + H(v) : v \in E\}$.

We highlight the following facts: 
\begin{itemize}
\item[(A)] We have that $\Uc_E$ is the set of all $E' \in \Gr_k(\R^m)$ intersecting $E^\perp$ transvsersally. 
In particular, $\Uc_E$ is open and dense for any $E \in\Gr_k(\R^m)$.

\item[(B)] We have the following basis-independent identification: 
\[
T_E \Gr_k(\R^m) = L(E, E^\perp) \, .
\]
If bases for $E, E^\perp$ are fixed, then 
we have the parametrization $\Uc_E \cong M_{m-k, k}(\R)$, the space of $(m-k) \times k$ 
real matrices.
\end{itemize}


With respect to the identification above, the Riemannian metric $g$ on $T_E \Gr_k (\R^m)$ can be expressed as 
\[
g_E(H_1, H_2) = \Tr_E(H_2^\top H_1) \, , 
\]
where $\Tr_E$ denotes the trace induced by the inner product $\langle \cdot, \cdot \rangle$ on $\R^m$ restricted to $E$. 
Recall that the orthogonal group $O(m)$ acts on $M = \Gr_k(\R^m)$ via the action $E \mapsto U (E)$ for $U \in O(m)$. It is standard that the orthogonal group acts isometrically on $(M, g)$.
In fact, $(M, g)$ is the unique (up to scalar) 
Riemannian metric on $M$ with respect to which $O(m)$ acts
isometrically. 
As usual, the Riemannian metric induces a volume measure 
$\Leb_{\Gr_k(\R^m)}$ and a geodesic distance $d_{geo}$ between subspaces in $\Gr_k(\R^m)$.

\subsection{Markov chain formulations; Lyapunov exponents
from stationary measures}\label{subsec:relevMarkovGrass}

Our random maps system $\{F^n_{\uo}\}_{n \geq 1}$ can be seen as a 
time-homogeneous Markov chain $Z_n := \{(x_n, y_n)\}$ given by 
\[
(x_n, y_n) = F^n_{\uo}(x_0, y_0) = F_{\omega_n}(x_{n - 1}, y_{n - 1})= R_{\omega_n}\circ F(x_{n - 1}, y_{n - 1})\, .
\]
That is to say, for fixed $\e$, the {\it transition probability} starting from $(x,y) \in \mathbb T^2$ is
\[
P((x,y), A) = Q(F z, A)
\] for Borel $A \subset\T^{2N}$, where the kernel $Q$ for the
randomizations $R_\omega$ is defined as in condition (E) above.  We write $P^{(k)}((x,y), \cdot)$ (or $P^{(k)}_{(x,y)}$) for the corresponding $k$-step transition probability. It is easy to see that for this chain, a measure $\nu$ is \emph{stationary},
meaning for any Borel set $A \subset \mathbb T^{2N}$,
\[
\nu(A) = \int P((x,y), A) \, d\nu(x,y) \, . 
\]

Since $F_\omega$ is always volume-preserving, it follows immediately that 
$\Leb_{\T^{2N}}$ is a stationary measure for $(Z_n)$. Equivalently, $\Leb_{\T^{2N}}$
gives rise to the $\tau$-invariant measure $\P \times \Leb_{\T^{2N}}$ on $\Omega \times \T^{2N}$, 
where the dynamical system $\tau$ on $\P \times \Leb_{\T^{2N}}$ is defined through
\[
\tau(\uo, z) = (\theta \uo, F_{\omega_1} z) \, .
\]

Let $(Z_n, E_n)$ denote the Markov process on the Grassmanian 
bundle $\Gr_N(\T^{2N}) \cong \T^{2N} \times \Gr_N(\R^{2N})$
defined for initial $(Z_0, E_0) \in \Gr_N(\T^{2N})$ by
\begin{gather*}
Z_n = F^n_\uo(Z_0) \, , \\
E_n = D_{Z_0} F^n_\uo (E_0) \, .
\end{gather*}
This gives rise to an associated dynamical system $\hat \tau : \Omega \times \Gr_N(\T^{2N}) \to 
\Omega \times \Gr_N(\T^{2N})$ defined by $$\hat \tau(\uo, z, E) = (\theta \uo, F_{\omega_1} z, 
D_z F_{\omega_1}(E)).$$ Recall that a measure of the form $\P \times \nu$ is $\hat \tau$-invariant iff
$\nu$ on $\Gr_N(\T^{2N})$ is stationary for $(Z_n, E_n)$, and ergodic iff $\nu$ is an ergodic 
stationary measure.

\begin{lem}
Assume condition $(E)$. Then, Lebesgue measure $\Leb_{\T^{2N}}$ is the unique, hence ergodic, stationary
measure for the Markov chain $(Z_n)$. 
\end{lem}
\begin{proof}The proof is similar to  Lemma 5 of \cite{BXY1}. Condition (E) implies the following:
\begin{itemize}
\item Every stationary measure for $(Z_n)$ is absolutely continuous w.r.t. Lebesgue measure. 
In particular, there are at most countably many distinct ergodic stationary measures, and 
giving rise to an at-most countable ergodic decomposition of $\T^{2N}$ (see, e.g., Kifer \cite{kifer1982perturbations})
\item Sufficiently close nearby points belong to the same ergodic component. 
\end{itemize}
It follows that there is exactly one ergodic component, which must coincide with Lebesgue
measure. 
\end{proof}

We are now in position to prove Lemma \ref{lem:existLE}, namely, that 
Lyapunov exponents $\lambda_i = \lim_n \frac{1}{n} \sigma_i(D_z F^n_\uo)$
exist and are constant for a.e. $z \in \T^{2N}$ and a.e. random sample $\uo$. 

\begin{proof}[Proof of Lemma \ref{lem:existLE}]
Recall that a measure of the form $\P \times m$ is $\tau$-invariant iff $m$ is stationary for
the Markov chain $(Z_n)$, and that $\P \times m$ is ergodic iff $m$ is an ergodic stationary measure (see Kifer \cite{kifer1982perturbations}).
Since $m = \Leb_{\T^{2N}}$ is ergodic, so is $\P \times m$. By the multiplicative ergodic theorem, 
it follows that Lyapunov exponents for the linear cocycle $D_z F_\uo^n$ over $\tau$ on $\Omega \times \T^{2N}$ are almost-surely constant with probability 1 for Leb-almost every $z \in \T^{2N}$. 
\end{proof}

The following relates stationary measures $\nu$ for the $(Z_n, E_n)$ process to Lyapunov exponents. 
\begin{lem}\label{lem:measureToLE}
Assume condition $($E$)$. 
Let $\nu$ be any stationary measure for $(Z_n, E_n)$ projecting to Lebesgue measure on $\T^{2N}$. Then, 
\begin{align}\label{eq:compareLEStatMeas}
\sum_{i = 1}^N \lambda_i \geq \E \int \log  \det(D_z F_\omega|_{E})  \, d \nu(z, E) \, .
\end{align}

\end{lem}
Above, for a $2N \times 2N$ matrix $A$ and $E \subset \R^{2N}, \dim E = N$, we write
$A|_E : E \to A(E)$ for the linear mapping of $E$ to $A(E)$ obtained by restricting $A$ to $E$. 
From this standpoint, $\det(A|_E)$ is defined as usual, e.g., as the volume ratio
\[
\det(A|_E) := \frac{\Leb_{A(E)} A(B_E)}{\Leb_E(B_E)} \, , 
\]
where $B_E \subset E$ is the unit ball, and $\Leb_E$ denotes Lebesgue measure on $E$. 

\begin{proof}
Without loss, we may assume $\nu$ is ergodic, hence $\P \times \nu$ is $\hat \tau$-ergodic. By the Birkhoff ergodic theorem applied to $\hat \tau$, we have 
\[
\int \log  \det(D_z F_\omega|_{E})  \, d \nu(z, E) = \lim_{n \to \infty} \frac{1}{n} \sum_{i = 0}^{n-1} \varphi \circ \hat \tau^i(\uo, z, E) =  \lim_{n \to \infty} \frac{1}{n} \log  \det(D_z F_{\uo}^n|_E)  
\]
for $\P \times \nu$-almost every $(\uo, z, E)$, where 
$
\varphi(\uo, z, E) := \log  \det(D_z F_{\omega_1}|_E) \, .
$
Recall that $\det (D_z F_\uo^n|_E)  \leq \prod_{i = 1}^N \sigma_i(D_z F^n_\uo)$. By Lemma \ref{lem:existLE}, we conclude
\[
 \int \log \det(D_z F_\omega|_{E})  \, d \nu(z, E) \leq \limsup_{n \to \infty} \frac{1}{n} \sum_{i = 1}^N \log \sigma_i(D_z F^n_\uo) = \sum_{i =1}^N \lambda_i \, . \qedhere
\]
\end{proof}

\begin{rmk}
In fact, equality holds in \eqref{eq:compareLEStatMeas} if 
$\nu \ll \mathfrak m$, the Riemannian volume on $\Gr_N(\T^{2N})$. When $\lambda_N > \lambda_{N + 1}$, this follows from: by the Multiplicity Ergodic Theorem,  
$\lim_n \frac{1}{n} \log \det(D_z F^n_\uo |_E) = \sum_1^N \lambda_i$
for a.e. $z \in \T^{2N}, \uo \in \Omega$ and for all $E$ transversal to the $N$-plane 
\[
E^{N + 1}_{(\uo, z)} := \{ v \in \R^{2N} : \lim_n \frac{1}{n} \log \| D_z F^n_\uo (v) \| \leq \lambda_{N + 1} \} \, .
\]
\end{rmk}

A key component of our analysis is the use of the nondegeneracy condition (ND) 
to provide a priori control on the density of stationary measures $\nu$ for the Grassmanian Markov chain $(Z_n, E_n)$. 
The following is an immediate consequence of (ND). 
\begin{lem}\label{lem:aPrioriBoundND}
Let $\nu$ be any stationary measure for $(Z_n, E_n)$. Then, $\nu \ll \mathfrak m$, where $\mathfrak m$ is the Riemannian
volume on $\Gr_N(\T^{2N})$, and satisfies
\[
\left\| \frac{d \nu}{d \mathfrak m} \right\|_{L^\infty} \leq M 
\]
where all notation is as in condition (ND).
\end{lem}

\begin{proof}
For $(z, E) \in \Gr_N(\T^{2N})$ and $K \subset \Gr_N(\T^{2N})$, define the transition kernels
\[
\hat P((z, E), K) = \P\big( (F_\omega z, D_z F_\omega(E)) \in K\big) = \P\big((Z_1, E_1) \in K | (Z_0, E_0) = (z, E) \big) \,,
\]
and note that
\[
\hat P((z, E), K) = \hat Q((F z, D_z F(E)), K)
\]
where $\hat Q$ is the kernel for $R_\omega$ as in \eqref{eq:defnhatQ}. In particular, by (ND), 
we have that $\hat P((z, E), \cdot) \ll \mathfrak m$, where $\mathfrak m$ is normalized Lebesgue measure on $\Gr_N(\T^{2N})$, while $d \hat P((z, E), \cdot) / d \mathfrak m = 
d \hat Q((F z, D_z F(E)), \cdot) / d \mathfrak m $ satisfies
\[
\left\| \frac{d \hat P((z, E), \cdot)}{d \mathfrak m} \right\|_{L^\infty}  
\leq M \, 
\]
uniformly in $(z, E) \in \Gr_N(\T^{2N})$. On the other hand, by stationarity, for $K \subset \Gr_N(\T^{2N})$ measurable we have
\[
\nu(K) = \int_{\Gr_N(\T^{2N})} \hat P((z, E), K) d \nu(z, E) 
\leq M \mathfrak m(K) \, .
\]
Therefore, $\nu \ll \mathfrak m$ and $d \nu / d \mathfrak m$ is essentially bounded from above by $M$. 
\end{proof}

\subsection{Hyperbolicity estimates assuming (F1), (F2)}

Let us record some estimates describing the quality of the predominant hyperbolicity
of the family $F = F_L$. Recall the notation
\[
\Cc^x_\a := \{ (u,v) \in \R^{2N} : \| v \| \leq \a \| u \| \} \, ,
\]
\[
B_\beta = \{ x \in \T^N :  |\det  D_x f_L | \leq L^{N - (1 - \beta)} \} \subset \T^N \, .
\]
We define $G_\beta = B_\beta^c$. 

\begin{lem}\label{lem:bulkHyp}
Fix $\beta \in (0,1)$ and let $L$ be sufficiently large. 
Let $z = (x,y) \in \T^{2N}$ be such that $x \in G_\beta$. 
\begin{itemize}
\item[(a)] Let $w = (u,v) \in T_z \T^{2N} \cong \R^{2N}$ be such that $w \in \Cc^x_{1/10}$. Then, 
$D_z F (w) \in \Cc^x_{1/10}$, and 
\[
\| D_z F (w) \| \geq 
 L^{\frac23 \beta} \| w \| \, .
\]
\item[(b)] Let $E \subset \R^{2N}$ be
an $N$-dimensional subspace such that $E \subset \Cc^x_{{1/10}}$. Then,
	\begin{itemize}
	\item[(i)] $E' := D_z F(E)$ is an $N$-dimensional subspace satisfying $E' \subset \Cc^x_{1/20}$, and
	\item[(ii)] $\det (D_z F|_E) \geq \frac{1}{2^N} L^{N - (1 - \beta)}$. 
	\end{itemize}
\end{itemize}
\end{lem}

\begin{proof}
For an $N \times N$ matrix $A$, 
write $m(A) = \| A^{-1} \|^{-1} = \min\{ \| A v\| / \| v \| : v \in \R^N \setminus \{ 0 \}\}$ for the minimum norm of $A$ (setting $m(A) = 0$ if $A$ is not invertible). To start, the estimate
\begin{align}\label{eq:minNormEst}
m(D_x f) \geq C_0^{- (N-1)} L^{\beta} 
\end{align}
follows from (F1) and the standard fact that $m(D_x f) \geq \det(D_x f) / \| D_x f\|^{N-1}$. 

For the estimate in (a), assume $w = (u,v) \in \Cc^x_\alpha$ for some $\alpha > 0$. Then
$\| v \| \leq \alpha \| u \|$ and $\| u \| \leq \| w \| \leq \sqrt{1 + \alpha^2} \| u \|$. 
So,
\begin{align*}
\| D_z F(w) \| & \geq \| D_x f(u)\| - \| v \| \geq (m(D_x f) - \alpha) \| u \| \\
& \geq \frac{C_0^{- (N-1)} L^\beta - \alpha}{\sqrt{1 + \alpha^2}} \| w \| \\
&\geq L^{2\beta / 3}\|w\|\,.
\end{align*}
The last inequality above holds when $L$ is large enough and $\alpha = 1/10$. 

For (b)(i): by hypothesis, we can express $E = \graph G = \{ (u, G(u)) : u \in \R^x\}$, 
where $G : \R^x \to \R^y$ is a linear map
with $\| G \| \leq 1/10$. To express $E'  = D_z F(E)$ in the form $E'= \graph G'$, we would need to have
that for all ${u'} \in \R^x$ there exists ${u} \in \R^x$ so that
\[
\left( \begin{array}{c} u' \\ G'(u') \end{array} \right) = 
\left( 
\begin{array}{c c} D_x f & -I_N \\ I_N & 0_N \end{array}
\right) 
\left( \begin{array}{c} u \\ G(u) \end{array} \right) 
= \left( \begin{array}{c} D_x f (u) - G(u) \\ u \end{array}  \right) 
\]
Formally, then, we ought to have $G' = (D_x f - G)^{-1}$. That this 
exists follows from \eqref{eq:minNormEst}; moreover, 
\[
\| G'\| \leq \frac{1}{m(D_x f) - 1/10} \leq 2L^{-\beta} \ll 1/20
\]
when $L$ is sufficiently large, hence $E' = D_z F(E) \subset \Cc^x_{1/20}$ as desired. 

For (b)(ii), define $\Pi^x : \R^{2N} \cong \R^x \times \R^y \to \R^x$ to be the orthogonal projection
onto $\R^x$. Then, 
\begin{align}\label{eq:detValueEstimate}
\det(D_z F|_E) = \det(D_x f - G) \cdot \frac{\det(\Pi^x|_E)}{\det(\Pi^x|_{E'})} 
= \det(D_x f - G) \cdot \frac{\det (I + G')}{\det (I + G)}
\end{align}
on noting that $(\Pi^x|_E)^{-1} = (I + G) : \R^x \to \R^{2N}$, and similarly for $\Pi^x|_{E'}$.
For these terms we have $ (1 - 1/10)^N \leq \det(I + G), \det(I + G')\leq  (1 + 1/10)^N]$, while
for the remaining $D_x f - G$ term we have
\[
\det(D_x f - G) \geq \det D_x f - \frac{\| G\|}{m(D_x f) - \| G \|} \geq \frac12 L^{N - (1- \beta)}
\]
using the elementary estimate $| \det(A + B) - \det(A)| \leq \| B \|/ (m(A) - \| B \|)$. 

\end{proof}

Condition (C) says that the randomizations $R_\omega, \omega \in \Omega_0$ do not 
`disrupt' the hyperbolicity of the system too much. The following is an immediate consequence of 
(C) and Lemma \ref{lem:bulkHyp}.

\begin{lem}\label{lem:horizPropagate}
The following holds for $\P_0$-a.e. $\omega \in \Omega_0$. Fix $\beta \in (0,1)$ and let $L$ be sufficiently large. Let $z = (x,y) \in \T^{2N}, x \in G_\beta$. Then, 
\begin{itemize}
\item[(i)] Let $w = (u,v) \in T_z \T^{2N} \cong \R^{2N}$ be such that $w \in \Cc^x_{1/10}$. 
Then, 
\[
\| D_z F_\omega(w)\| \geq L^{\frac12 \beta} \| w \| 
\]
\item[(ii)] Let
$E \subset \Cc^x_{1/10}$ be an $N$-dimensional subspace. Then, $E' = D_z F_\omega(E)$ is an $N$-dimensional subspace with $E' \subset \Cc^x_{1/10}$. 
\end{itemize}

\end{lem}

\section{Proof of Theorem \ref{thm:main}}\label{SThmMain}

In brief, our method will be to obtain a lower bound of the form 
$\sum_{i = 1}^N \lambda_i \geq (1 - \varepsilon) N \log L$ on 
the sum of the first $N$ Lyapunov exponents for $\varepsilon > 0$ small and $L$ sufficiently large. 
This directly implies $\lambda_i \geq (1 - (2 N-1) \varepsilon) \log L$ for each $1 \leq i \leq N$,
in view of the fact that $\lambda_i \leq \lambda_1 \leq (1 + \varepsilon) \log L$ for all 
$i$, $\varepsilon > 0$ and $L \gg 1$ (see condition (F1)). Since $\sum_{i = 1}^{2N} \lambda_i = 0$, 
similar considerations apply to the exponents $\lambda_{N + 1}, \cdots, \lambda_{2N}$. 
These proofs are straightforward and omitted for brevity. 

From this point forward, we will focus our attention on $\sum_{i = 1}^N \lambda_i$, which we shall 
estimate using a stationary measure $\nu$ for the Markov chain $(Z_n, E_n)$
on $\Gr_N(\T^{2N})$, following Lemma \ref{lem:measureToLE}. 
Applying \eqref{eq:compareLEStatMeas}, condition (C), and the chain rule $D_z F_\omega = D_{F_\omega z} R_\omega \circ D_z F$, we have
\[
 \sum_{i = 1}^N \lambda_i \geq - N \log 2 + \E \int_{\Gr_N(\T^{2N})} \log | \det(D_z F|_{E}) | \, d \nu(z, E)
\]
By stationarity, for any bounded measurable $\phi : \Gr_N(\T^{2N} ) \to \R$ we have (by a slight abuse of notation)
\[
\int \phi(z, E) d \nu(z, E) = \int \left(\E_{(z, E)} \phi(Z_n, E_n)\right) d \nu(z, E) = \E \left[\int \phi(Z_n, E_n) d \nu(z, E)\right]
\]
 for all $n \geq 1$. Above, we interpret $(Z_n, E_n)$ as a function of the initial
condition $(Z_0, E_0) = (z, E)$ and the random sample $\uo$, with $\E_{(z, E)}$ denoting the corresponding expectation.
 Applying to $\phi(z, E) = \log |\det(D_z F|_E)|$, we conclude
\[
 \sum_{i = 1}^N \lambda_i  \geq - N \log 2 + \E \underbrace{\int \log | \det(D_{Z_n} F|_{E_n})| d \nu(z, E)}_{(**)} \, .
\]

To prove Theorem \ref{thm:main} it therefore suffices to bound $(**)$ as follows.
\begin{prop}[Main estimate]\label{prop:mainEst}
Fix $\alpha, \beta \in (0,1)$ and {$\d \in(0,c_\beta)$}. Let $L$ be sufficiently large in terms of these parameters. 
Then, there exists $n \gg 1$, depending on $L$, such that for a.e. $\uo \in \Omega$, we have
\[
(**) = \int \log | \det(D_{Z_n} F|_{E_n}) | d \nu(z, E) \geq \alpha N \log L \, .
\]
\end{prop}
The proof of Theorem  \ref{thm:main} is complete 
upon adjusting the parameter $\alpha$ and taking $L$ large enough to absorb the
remaining additive term $- N \log 2$.

\subsection*{Proof of Proposition \ref{prop:mainEst}: Exploiting predominant hyperbolicity}

Below $n \geq 1$ is fixed, to be determined later, and $\uo \in \Omega$ is an arbitrary random sample. 
Recall that $D_z F_\omega$ is strongly expanding in the $G_\beta$ along which $D_z F_\omega$ is strongly expanding in the 
horizontal cone $\Cc_\alpha^x = \{ ( u, v) : \| v \| \leq \alpha \| u \|\}$
for $z \in G_\beta$ (Lemma \ref{lem:bulkHyp}). For $n \geq 1$, define
\[
G^n_\beta = \{ z \in \T^{2N} : Z_i \in G_\beta \text{ for all } 0 \leq i \leq n -1\} 
\]
to be the set of trajectories experiencing this hyperbolicity for $n$ timesteps, where as usual we 
condition on $Z_0 = z$. 

Fix $z \in G^n_\beta$. Hyperbolic expansion along the $x$-direction $\R^x$ implies that 
the `bulk' of Grassmanian dynamics is attracted to a close vicinity of $\R^x$. This is, after all, 
the conceptual picture underlying the $N = 1$ case studied in the previous paper \cite{BXY1}.
The following is the analogue of Lemma 10 in \cite{BXY1}. 

\begin{prop}\label{prop:estLebMass}
Let $\uo \in \Omega$ be arbitrary, and let $\beta \in (0,1), n \geq 1$. Fix $z \in G^n_\beta$. 
Set $E_n = D_z F_\uo^n(E)$. Then,
\[
\Leb_{\Gr_N(\R^{2N})} \{ E \in \Gr_N(\R^{2N}) \text{ such that } E_n \notin \Cc^x_2\} \leq L^{- \beta n} \, .
\]
\end{prop}

The proof of Proposition \ref{prop:estLebMass} is deferred for now. Let us show how it can be used to prove Proposition \ref{prop:mainEst}. For $z \in G^n_\beta$ define $\Gc_z^n = \{ E \in \Gr_N(\R^{2N}) : E_n \in \Cc_2^x\}$. Letting $\beta^* \in (0,1)$ be a parameter to be chosen later, define
\[
\Gc^n = \{ (z, E) \in \Gr_N(\T^{2N}) : z \in G^n_\beta \cap (F^n_\uo)^{-1} G_{\beta^*}, E \in \Gc_z^n \} \,  
\]
and $\Bc^n = \Gr_N(\T^{2N}) \setminus \Gc^n$. The integral of $(**)$ along $(z, E) \in \Gc^n \subset \Gr_N(\T^{2N})$ will result in a tight lower bound for $\det (D_{Z_N} F|_{E_n})$, 
while $\Bc^n$ is an error set along which we use the poor estimate
\begin{align}\label{eq:badEstDeterminant}
\log |\det (D_{Z_n} F|_{E_n})| \geq - N \log (2 C_0 L) \,,
\end{align}
which follows from (F1) and the form of the mapping $F = F_L$. 

Splitting
$(**)$ along the partition $\Gc^n, \Bc^n$, we have
\[
\int_{\Gc^n} \log | \det(D_{Z_n} F|_{E_n}) | d \nu(z, E) 
\geq
(1 - \nu(\Bc^n)) \underbrace{\inf_{z \in G_{\beta^*}, E \in \Cc^x_2} \log | \det (D_{z} F|_E)|}_{\dagger} \, . 
\]
Choosing $\beta^*$ sufficiently close to $1$, we can arrange for $\dagger \geq \frac{1 + \alpha}{2} N \log L$ (Lemma \ref{lem:bulkHyp}(b)(ii)) on taking $L$ sufficiently large. Plugging in \eqref{eq:badEstDeterminant}, we obtain
\[
(**) \geq \frac{1 + \alpha}{2} N \log L -  2 \nu(\Bc^n) \cdot N \log (2 C_0 L) \, .
\]

It remains to bound $\nu(\Bc^n)$ from above. We decompose $\Bc^n = \Bc^{n, 1} \cup \Bc^{n,2}$, where
\begin{gather*}
\Bc^{n,1} = \big( (G^n_\beta)^c \cup (F^n_\uo)^{-1} G_{\beta^*} \big) \times \Gr_N(\R^{2N})  \\
\Bc^{n, 2} = \{ (z, E) : z \in G^n_\beta, E_n \notin \Cc^x_2\} = \{ (z, E) : z \in G^n_\beta, E \notin \Gc^n_z \} \, .
\end{gather*}
For $\Bc^{n, 1}$ we have the simple estimate
\[
\nu(\Bc^{n, 1}) = \Leb_{\T^{2N}}((G^n_\beta)^c \cup (F^n_{\uo})^{-1} G_{\beta^*} ) \leq n C_\beta L^{- c_\beta} + C_{\beta^*} L^{- c_{\beta^*}} \, .
\]
For $\Bc^{n, 2}$, we estimate
\[
\nu(\Bc^{n, 2}) \leq M \Leb_{\Gr_N(\T^{2N})} (\Bc^{n, 2}) \leq M L^{- \beta n}
\]
using our bound on $\frac{d \nu}{d \mathfrak m}$ from Lemma \ref{lem:aPrioriBoundND}
and the estimate in Proposition \ref{prop:estLebMass}.

In total, we have shown that
\begin{align*}
(**) & \geq \frac{\a +1}{2} N \log L - \log L^N \bigg(  n C_\beta L^{- c_\beta} + M L^{- \beta n} + C_{\beta^*} L^{- c_{\beta^*}} \bigg) 
\end{align*}

Fix $n = \lceil L^{c_\beta - \delta} \rceil$ for some small $\delta \ll c_{\beta}$. Then, 
$n C_\beta L^{- c_\beta} = O(L^{- \delta})$, while $M L^{- \beta n} \lesssim L^{- \delta}$ 
as long as
\[
 M \leq L^{\frac12 \beta L^{c_\beta - \d}} \, .
\]
Thus, under this condition relating $M$ and $L$, we have
\[
(**) \geq \a N \log L +\left(  \frac{1 - \a}{2}  - C L^{- \min\{ \d, c_{\b^*}\}} \right) N \log L 
\geq \a N \log L
\]
assuming $L$ is sufficiently 
large in terms of $\alpha, \beta, \beta^*, \delta$. This completes the proof of Proposition \ref{prop:mainEst}.

\section{Proof of Main Proposition (Proposition \ref{prop:estLebMass})}\label{SProp}

In Section \ref{subsec:grassGeo} we recall and prove some necessary 
facts concerning the geometry of Grassmanians. The proof of Proposition
\ref{prop:estLebMass} is carried out in Sections \ref{subsec:SVD} and 
\ref{subsec:proveMainProp}.

\subsection{Lemmas on Grassmanian Geometry}\label{subsec:grassGeo}

In what follows, given a subspace $E \subset \R^m$, we write 
$\Pi_E : \R^m \to E$ for its corresponding orthogonal projection.

The following alternative metric $d_H$ on Grassmannians is very useful in practice. 
\begin{defn}
Let $E, E' \in \Gr_k(\R^m)$. We define the \emph{Hausdorff distance} $d_H(E, E')$ between
them by
\[
d_H(E, E') = \max\left\{ \max_{\substack{v' \in {E'} \\ \| v' \| = 1}} d(v', E), \max_{\substack{v \in {E} \\ \| v \| = 1 }} d(v, E') \right\} \, , 
\]
where above $d(v, E)$ denotes the minimal Euclidean distance between $v \in \R^m$ and $E \subset \R^m$. 
\end{defn}

The distance function $d_H$ is uniformly equivalent to the geodesic distance $d_{geo}$: 
\begin{lem}\label{lem:unifEquivDistances}
For any $E, E' \in \Gr_k(\R^m)$, we have
\[
\frac{2}{\pi} d_{geo}(E, E') \leq d_H(E, E') \leq d_{geo}(E, E')
\]
\end{lem}
This appears to be well-known, but we are unable to find a proof of Lemma 
\ref{lem:unifEquivDistances} in the literature. For the 
sake of completeness a sketch is provided below. 
\begin{proof}
Let $E, E' \in \Gr_k(\R^m)$. Then, 
$
d_{geo}(E, E') = \sqrt{ \psi_1^2 + \cdots + \psi_k^2}$
where each $\psi_i = \psi_i(E, E') \in [0,\pi/2]$ is the $i$-th \emph{Jordan angle} between $E, E'$, defined by, e.g., 
\begin{align}\label{eq:minMaxJordanAngle}
\cos \psi_i = \min_{\substack{P \subset E \\ \dim P = i}} \max_{\substack{v \in P \\ \| v \| = 1}} 
\max_{\substack{w \in E' \\ \| w \| = 1}} \langle v, w\rangle
\end{align}
(see Proposition 3(b) of \cite{neretin2001jordan}). We have $\psi_1 \leq \psi_2 \leq \cdots \leq \psi_k$, hence $\psi_k \leq d_{geo}(E, E') \leq k \psi_k.$

To connect this with the Hausdorff metric, by \cite{kato2013perturbation} Theorem I-6.34
and some elementary arguments, we have
\begin{align*}
d_H(E, E') &= \| (I - \Pi_{E'}) \Pi_E\| = \sup_{v \in E} \| (I - \Pi_{E'}) v\| = \sup_{v \in E} d(v, E') \\
& = \sin \angle (v, \Pi_{E'} v) \, .
\end{align*}
On the other hand, by \eqref{eq:minMaxJordanAngle}, we have
\[
\max_{\substack{w \in E' \\ \| w \| = 1}} \langle v, w \rangle = \left\langle v, \frac{\Pi_{E'} v}{\|\Pi_{E'} v\|} 
\right\rangle = \cos \angle (v, \Pi_{E'} v) \, , 
\]
hence $\psi_k = \max_{v \in E, \| v \| = 1} \angle (v, \Pi_{E'} v)$. We conclude, then, that
$d_H(E, E') = \sin \psi_k.$
In particular, $\frac{2}{\pi} \psi_k \leq d_H(E, E') \leq \psi_k$. This completes the proof. 
\end{proof}

We close this section with a geometric description of the set
$(\Uc_E)^c$ of $k$-dimensional subspaces meeting $E^\perp$ nontransversally. 

\begin{lem}\label{lem:complementChartDescription}
Assume $k \leq m / 2$. 
Then, the set $(\Uc_E)^c$ is a finite union of closed submanifolds of $\Gr_k(\R^m)$
of codimension $\geq 1$. 
\end{lem}
Note that in particular, the Lebesgue measure of $(\Uc_E)^c$ is zero. 

\begin{proof}

Given $E \in \Gr_k(\R^m)$, write $V = E^\perp$ and define
$\Vc = (\Uc_E)^c$, which by point (A) at the beginning of Section \ref{SProp}
 is the set of $k$-dimensional subspaces
intersecting $V$ nontransversally. We will describe $\Vc$ as the image of a 
fiber bundle $\Ec$, to be defined below, via a smooth mapping
$\Phi : \Ec \to \Gr_k(\R^m)$. As we will show, $\dim \Ec < k (m-k) = \dim \Gr_k(\R^m)$, hence
$\Vc$ can be covered by embedded submanifolds of dimension $< k (m-k)$. 

To define $\Phi$ and $\Ec$, we first introduce some notation. 
Given $v \in \R^m$ let $I_v = \{ S \in \Gr_k(\R^m) : v \in S\}$. 
Then, each $S \in I_v$ is uniquely specified by a corresponding $k-1$-dimensional subspace
$S_v := S \cap \langle v \rangle^\perp = (I - \Pi_v) (S)$, where $\Pi_v : \R^m \to \langle v \rangle$ is the orthogonal projection. So, we can (canonically) identify
$I_v \cong \Gr_{k-1}(\langle v \rangle^\perp)$. The latter is essentially $\Gr_{k-1}(\R^{m-1})$ and has dimension $(k-1)(m - 1 - (k-1)) = (k-1)(m -k)$. 

Let $\pi : \mathcal E \to \Gr_1(V)$ denote the fiber bundle
over $\Gr_1(V)$ with fibers $\Gr_{k-1}(\langle v \rangle^\perp)$. Write elements
of $\Ec$ as $(v, \hat S)$, where $v \in V, \hat S \in \Gr_{k-1}(\langle v \rangle^\perp)$. 
We define $\Phi : \mathcal E \to \Gr_m(\R^k)$ to be the sum of subspaces
\[
\Phi(v, \hat S) = \langle v \rangle + \hat S
\]
in $\R^m$. Evidently, the image of $\Phi$ coincides with $\Vc$. Since
$\dim \Ec = (k-1) + (k-1)(m - k) = (k-1)(m - (k - 1))$, it follows that 
$\Vc$ can be covered by finitely many closed submanifolds of dimension
$\leq (k-1)(m - (k-1)) < k (m - k)$. 
\end{proof}

\subsection{Singular value decomposition} \label{subsec:SVD}

Throughout, the parameter $\beta \in (0,1)$ is fixed, as are $n \geq 1$, $\uo \in \Omega$ and
 $z \in G_\beta^n$. We now proceed to study the singular-value decomposition for the iterated Jacobian
$D_z F^n_\uo$. 

\begin{lem}\label{lem:SVD}
Let $\sigma_i =\sigma_i(D_{z} F^n_\omega), \sigma_1 \geq \sigma_2 \geq \cdots \geq \sigma_{2N}$ denote the singular values of $D_{z} F^n_\omega$.
\begin{itemize}
\item[(i)] We have 
\[
\sigma_N \geq L^{n \beta/2} \geq L^{- n \beta/2} \geq \sigma_{N + 1}
\]
\item[(ii)] Let $h_1, \cdots, h_{2N}, h_1', \cdots, h_{2N}'$ denote the orthogonal 
  bases of $\R^{2N}$ for which 
  \[
  D_{z} F^n_\uo h_i = \sigma_i h_i'
  \]
Then, $h_i , h_i' \in \Cc^x_{1/10}$ for $1 \leq i \leq N$ and $h_i, h_i' \in \Cc^y_{1/10}$ for all $N + 1 \leq i \leq 2N$.
\end{itemize}
\end{lem}
For $\alpha > 0$, we have written 
$\Cc^y_\alpha = \{ (u,v) \in \R^{2N} : \| u \| \leq \alpha \| v \|\}$ for the cone of
vectors roughly parallel to $\R^y$. 

\begin{proof}
It follows from Lemma \ref{lem:horizPropagate}(ii) that for any $E \in \Gr_N(\R^{2N}), E \subset \Cc^x_{1/10}$ that
\[
D_{Z_0} F^n_\uo ( E) \subset \Cc^x_{1/10} \, .
\]
A mild variation of the arguments for Lemma \ref{lem:horizPropagate} similarly implies that
 $(D_{Z_0} F^n_\uo)^\top (E) \subset \Cc^x_{1/10}$. The same then holds for $(D_{Z_0} F^n_\uo)^\top D_{Z_0} F^n_\uo$ and $D_{Z_0} F^n_\uo (D_{Z_0} F^n_\uo)^\top$. 
 Since each of the latter two are self-adjoint, it follows that
the top $N$ eigenvectors for $(D_{Z_0} F^n_\uo)^\top D_{Z_0} F^n_\uo$ (resp. $D_{Z_0} F^n_\uo (D_{Z_0} F^n_\uo)^\top$) span an $N$-dimensional subspace $E \subset \Cc_{1/10}^x$ (resp. $E' \subset \Cc^x_{1/10}$); see Lemma \ref{coneSVD} in Appendix \ref{sec:SVDapp}. 
Moreover, we have $D_{Z_0} F^n_\uo(E)  =E'$. 
That the remaining eigenvectors of 
$(D_{Z_0} F^n_\uo)^\top D_{Z_0} F^n_\uo$ (resp. $D_{Z_0} F^n_\uo (D_{Z_0} F^n_\uo)^\top$)
lie in $\Cc^y_{1/10}$ now follows. 

Lastly, for all $v \in \Cc^x_{1/10}$, we have $\| D_{Z_0} F^n_\uo v\| \geq L^{n \beta/2}$ 
by Lemma \ref{lem:horizPropagate}(i), which yields the estimate
$\sigma_N(D_{Z_0} F^n_\uo) \geq L^{n \beta / 2}$. Similarly, as one can check, 
$\| (D_{Z_0} F^n_\uo)^{-1} w\| \geq L^{n \beta / 2}$ for all $w \in \Cc^y_{1/10}$,
hence $\sigma_{N + 1}(D_{Z_0} F^n_\uo) \leq L^{- n \beta / 2}$. 
This completes the proof.
\end{proof}

\subsection{The proof of Proposition \ref{prop:estLebMass}}\label{subsec:proveMainProp}

\newcommand{\Span}{\operatorname{Span}}

Define
\begin{align*}
\Hc = \Hc(D_z F^n_\uo) = \Span\{ h_1, \cdots, h_N\} \\
\Hc' = \Hc'(D_z F^n_\uo) = \Span\{ h_1', \cdots, h_N'\} \, , 
\end{align*}
noting that
$\Hc^\perp = \Span\{ h_{N + 1}, \cdots, h_{2N}\}, 
(\Hc')^\perp = \Span\{ h_{N + 1}', \cdots, h_{2N}'\}$. 
By Lemma \ref{lem:SVD}, we have that $\Hc, \Hc' \subset \Cc^x_{1/10}$. 
Below, given $\eta > 0$ and $S \subset \Gr_N(\R^{2N})$, 
we write $\Nc_\eta(S)$ for the (open) 
$\eta$-neighborhood of $S$ with respect to the geodesic distance $d_{geo}$. 

\begin{lem}\label{lem:containmentGrassman}
There exists a universal constant $c > 0$ depending only on $N$ such that
\begin{align}
\{ E \in \Gr_N(\R^{2N}) : D_z F^n_\uo(E) \text{ is not contained in } \Cc^x_2\} 
\subset
\Nc_\eta((\Uc_\Hc)^c) \, , 
\end{align}
where $\eta = c L^{- \beta n}$. 
\end{lem}
Proposition \ref{prop:estLebMass} follows, since 
$(\Uc_\Hc)^c$ is the finite union of a collection of closed 
submanifolds of $\Gr_N(\R^{2N})$  
(Lemma \ref{lem:complementChartDescription}). Here, we use the standard
fact that if $M' \subset M$ is a closed submanifold of a compact 
Riemannian manifold $M$ with strictly positive codimension, then the Lebesgue measure 
of any neighborhood $\Nc_\eta(M')$ is $\leq C \eta$, where $C > 0$ depends only on $M$. 

\begin{proof}[Proof of Lemma \ref{lem:containmentGrassman}]
Let $E \in \Gr_N(\R^{2N})$ be such that $E' := D_x F^n_\uo(E)$ is not contained in 
$\Cc^x_2$. Let $v' \in E' \setminus \Cc^x_2$; since $\Cc^x_2$ is a cone and $E'$ is a subspace,
we can assume without loss that $v'$ is a unit vector. 

Now, let 
\[
v' = v'_{\|} + v'_\perp \, , 
\]
where $v'_\| \in \Hc', v'_\perp \in (\Hc')^\perp$. Since $v' \notin \Cc^x_2$ and 
$\Hc' \subset \Cc^x_{1/10}$, it follows that $\| v'_\perp\| \approx 1$.
Now, let $v = (D_z F_\uo^n)^{-1} (v') / \|(D_z F_\uo^n)^{-1} (v')\|$, so that
\[
v = v_\| + v_\perp = \frac{(D_z F^n_\uo)^{-1}(v'_\|)}{\| (D_z F^n_\uo)^{-1}(v')\|} + 
\frac{(D_z F^n_\uo)^{-1}(v'_\perp)}{\| (D_z F^n_\uo)^{-1}(v')\|} 
\]
since $D_z F^n_\uo(\Hc) = \Hc', D_z F^n_\uo(\Hc^\perp) = (\Hc')^\perp$
(see Appendix \ref{sec:SVDapp}). 
To estimate these components, we have
\begin{gather*}
\| (D_z F^n_\uo)^{-1} (v'_\|)\| \leq (\sigma_N(D_z F^n_\uo))^{-1} \| v'_\| \| \leq 
L^{- \beta n/2}  \, , \\
\| (D_z F^n_\uo)^{-1} (v'_\perp)\| \geq (\sigma_{N + 1}(D_z F^n_\uo))^{-1} \| v'_\perp\| 
 \gtrsim L^{\beta n/2}
\end{gather*}
using that $\| v_\perp'\| \approx 1$. Continuing, 
\begin{align*}
 \| (D_z F^n_\uo)^{-1} (v')\| \geq \| (D_z F^n_\uo)^{-1} (v'_\perp)\| - \| (D_z F^n_\uo)^{-1} (v'_\|)\| 
 \gtrsim L^{\beta n/2} - L^{- \beta n/2} \gtrsim L^{\beta n/2} \, .
\end{align*}
Thus, we conclude that $\| v_\| \| \lesssim L^{-  \beta n}$, while
$1 - \| v_\perp\| = O(L^{-  \beta n})$. This immediately implies
that 
\begin{align}\label{eq:distHcPerp}
d(v, \Hc^\perp) \leq \left\| v - \hat v_\perp \right\| = O(L^{-  \beta n}) \, ,
\end{align}
where $\hat v_\perp = \frac{v_\perp}{\| v_\perp\|} \in \Hc^\perp$ is a unit vector. 

Fix a basis $w_1, \cdots, w_{N-1}$ for the orthogonal complement of $v$ in $E$
and define $$\hat E := \Span\{ \hat v, w_1, \cdots, w_{N-1}\}.$$ It follows from 
\eqref{eq:distHcPerp} that $d_H(\hat E, E) \lesssim L^{-  \beta n}$. Since 
$\hat E \in (\Uc_\Hc)^c$ and $d_H, d_{geo}$ are uniformly equivalent
by Lemma \ref{lem:unifEquivDistances}, we conclude that
$d_{geo}(E, (\Uc_{\Hc})^c) \lesssim L^{- \beta n}$. 
\end{proof}

\section{Proof of Theorems \ref{ThmStd} and \ref{thm:strongCoupling}}
\label{sec:app}

\subsection{Proof of Theorem \ref{ThmStd}}
It suffices to check condition (F2) holds for the family
\[
f_L(x_1, \cdots, x_N) = ( 2 x_i + L \sin(2 \pi x_i) + \sum_{j \neq i} \mu_{ij} \sin 2 \pi (x_j - x_i) \, , 
\]
where $(\mu_{ij})$ is a fixed family of coefficients. 

Write $f = (f^1, \cdots, f^N)$ in component form. 
By the Leibniz formula for the determinant,
\[
\det L^{-1} D_x f = L^{-N} \sum_{\tau \in S_N} \prod_{i = 1}^N \frac{\pd f^i}{\pd x_{\tau(i)}}
\]
where the outer summation is over the permutations $\tau \in S_N$ of the set $\{ 1, \cdots, N\}$. The dominant term
is the product $\prod_{i = 1}^N L \cos 2 \pi x_i$; precisely, 
\[
\det L^{-1} D_x f - (2 \pi)^N \prod_{i =1}^N \cos 2 \pi x_i = O(L^{-1})
\]
when $L$ is taken sufficiently large relative to $\max_{ij} \mu_{ij}$. Therefore
\[
B_\beta \subset \left\{ \left| \prod_{i = 1}^N \cos 2 \pi x_i\right|  \leq L^{2\beta - 1}\right\}
\]
for $L$ sufficiently large. To estimate the volume of the RHS, we bound $|\cos 2 \pi x| \geq 4 \min\{ |x - 1/4|, |x - 3/4|\}$, so that
$|\prod_{i = 1}^N \cos 2 \pi x_i | \geq \min\{ \prod_{i = 1}^N | x_i - r_i | \}$, where the minimum is taken over all possible configurations of
$r_i \in \{ 1/4, 3/4\}, 1 \leq i \leq N$. Thus,
\[
B_\beta \subset \bigcup \left\{ \prod_{i = 1}^N | x_i - r_i | \leq L^{2 \beta - 1} \right\}
\]
where the union is again over all possible configurations of the $r_i$. We conclude that 
\[
\Leb B_\beta \lesssim \Leb \{ (x_1, \cdots, x_N) \in [0,1]^N : \prod_i x_i \leq L^{2 \beta  - 1} \} \, .
\]
For the RHS, we have the asymptotic $\lesssim L^{3 \beta - 1}$ (see Lemma \ref{lem:areaComputation} below). Therefore, condition (F2)
is satisfied with $c_\beta = 1 - 3 \beta$. 

\begin{lem}\label{lem:areaComputation}
Define $S_N(\d)  = \{ (x_1, \cdots, x_N) \in [0,1]^N : \prod_i x_i \leq \d\}$. Then, For any $\d > 0$, we have
\[
\Leb S_N(\d) = \d \sum_{i = 0}^{N-1} \frac{( - \log \d)^i}{i!}
\]
\end{lem}
\begin{proof}
Define
\[
J = J(x_1, \cdots, x_N) := \min\{1 \leq j \leq N : 0 \leq x_j \leq \frac{\d}{x_1 \cdots x_{j-1}}\} \, ,
\]
and note that $J$ is defined on $S_N(\d)$. In particular, 
\[
\{ J = j \} \cap S_N(\d) = \left\{ \frac{\d}{\prod_{i \leq \ell-1} x_i} \leq x_\ell \leq 1 \text{ for all } 1 \leq \ell \leq j-1, 
0 \leq x_j \leq \frac{\d}{x_1 \cdots x_{j-1}}
\right\}
\]
with no constraint on $x_{j + 1}, \cdots, x_N$. Thus $\Leb(S_N(\d) \cap \{ J = j \}) = \Leb(S_j(\d) \cap \{ J = j\})$ for all $j$.
It therefore suffices to compute $\Leb(S_N(\d) \cap \{ J = N\})$. 
This is given by
\[
\int_{x_1 = \d}^1 \int_{x_2 = \frac{\delta}{x_1}}^1 \cdots \int_{x_{N-1} = \frac{\delta}{x_1 \cdots x_{N-2}}}^1 \frac{\delta}{x_1 \cdots x_{N-1}}
dx_{N-1} \cdots dx_1
\]
For $c \in (0,1), k \geq 0$, we have
\[
\int_{y = c}^1 \frac{c}{y} \left( \log \frac{c}{ y} \right)^k dy = -\int_{\log \frac{c}{y} = 0}^{\log c} \left( \log \frac{c}{ y} \right)^k d \log \frac{c}{y}
= - \frac{c}{k +1} (\log c)^{k+1} \,. 
\]
Thus, after $k$ iterated integrals, we have
\begin{gather*}
\int_{x_{N-k} = \frac{\d}{x_1 \cdots x_{N-k-1}}}^1 \cdots \int_{x_{N-1} = \frac{\d}{x_1 \cdots x_{N-2}}}^1 \frac{\d}{x_1 \cdots x_{N-1}} dx_{N-1} \cdots dx_{N-k}
= \\ 
\frac{(-1)^{k}}{k!} \frac{\d}{x_1 \cdots x_{N-k-1}} \left( \log \frac{\d}{x_1 \cdots x_{N-k-1}} \right)^{k}
\end{gather*}
and after $k = N-1$ such integrals, we deduce $\Leb( S_N(\d) \cap \{ J = N\}) = \frac{\d}{(N-1)!} ( - \log \d)^{N-1}$. This completes the proof. 
\end{proof}

\subsection{Transversality criterion and genericity for (F2) }
\subsubsection{Transversality criterion for (F2)}
Below, we derive the general transversality criterion 
\eqref{eq:transConditionIntro} for condition (F2) 
for families of the form $f_L = L \psi + \varphi$. 

\begin{lem}[Transversality Criterion]\label{lem:F2Generic}
Let $\psi, \varphi : \T^N \to \R^N$ be $C^2$ mappings and assume that 
\begin{align}\label{eq:psiTransCond}
\{ \det D_x \psi = 0 \} \cap \{ D_x \det D \psi = 0 \} = \emptyset \, .
\end{align}
Then, $f_L = L \psi + \varphi$ satisfies condition (F2) with $c_\beta = 1-\beta$ 
for all $L$ sufficiently large. Precisely, 
for any $\beta > 0$, there exists $C_\beta = C_\beta(\psi, \varphi)$ 
so that for any $L$ sufficiently large (in terms of $\psi, \varphi$), we have that
\[
\Leb \{ \det (L^{-1} D_x f_L) \leq L^{- (1 - \beta)} \} \leq C_\beta  L^{- (1 - \beta)} \, . 
\]
\end{lem}

\begin{rmk}
Note that $\psi(x) = (\sin 2 \pi x_i)_i$ does \emph{not} satisfy the
transversality condition \eqref{eq:psiTransCond}; this is why we had to check
(F2) by hand in the proof of Theorem \ref{ThmStd}. However, \eqref{eq:psiTransCond}
does hold for a large class of models: as we check below in Proposition \ref{prop:F2genericN2},
it is satisfied by a $C^2$ generic set of $\psi$. 
\end{rmk}

\begin{proof}
We begin with the following straigthforward consequence of the constant rank theorem applied 
$x \mapsto \det D_x \psi$: there exist $\hat C > 0$ with the property that for any
$0 \leq \epsilon \leq \hat \epsilon(\psi)$, we have
\[
\Leb\{ \det D_x \psi \leq \epsilon \} \leq \hat C  \epsilon \, .
\]
We will also need the following estimate: if $A, B$ are $N \times N$ matrices,
then there exists $C_{A,B} > 0$, depending only on $\max |A_{ij}|, \max |B_{ij}|$, and $N$, such that
\[
\det(A + \eta B) \geq \det(A) - C_{A,B} \eta \, .
\]
This can be obtained, e.g., from the formula
\[
\det(A + \eta B) - \det(A) = \int_0^\eta \Tr\big( \operatorname{Adj}(A + s B) B \big) ds \, ,
\]
where $\operatorname{Adj}(\cdot)$ denotes the adjugate of a square matrix. 
With this notation, define $\tilde C = \sup_x C_{D_x \psi, D_x \varphi}$.

To complete the proof, let $x \in \T^N$ be such that $\det D_x \psi \geq 2 L^{- (1 - \beta)}$. 
Then, 
\[
\det D_x (\psi + L^{-1} \phi) \geq \det D_x \psi - \tilde C L^{-1}\geq L^{- (1 - \beta)}
\]
if $L^\beta \gg \tilde C$. Thus, $\{ \det D_x \psi \geq 2 L^{- (1 - \beta)}\} \subset
\{ \det (L^{-1} D_x f_L) \geq L^{- (1 - \beta)}\}$. Taking complements, 
we conclude that 
\[
\Leb\{ \det (L^{-1} D_x f_L) \leq L^{- (1 - \beta)} \} 
\leq \Leb\{ \det D_x \psi \leq 2 L^{- (1 - \beta)}\} \leq 2 \hat C L^{- (1 - \beta)}
\]
on taking $L$ large enough so that $2 L^{- (1 - \beta)} \ll \hat \epsilon(\psi)$. The proof
is complete on setting $C_\beta = 2 \hat C$. 
\end{proof}
\subsubsection{Genericity of (F2) when $N = 2$}

In this subsection, we consider genericity of the transversality condition \eqref{eq:psiTransCond} 
used to prove property (F2). We expect that \eqref{eq:psiTransCond} is generic in general. 
For simplicity, however, we prove this only in the special case $N = 2$. 

 \begin{prop}\label{prop:F2genericN2}There is a residual set $\mathcal R$ in $C^r(\T^2,\R^2),\ r\geq 2$ such that for all $\psi\in \mathcal R$, equation \eqref{eq:psiTransCond} holds, i.e., we have  that $0$ is a regular value of $x \mapsto \det D_x\psi$.
\end{prop}
\begin{proof}
We write $\psi=(\psi_1,\psi_2)$ where $\psi_1,\psi_2:\ \T^2\to\R$. It is well-known that Morse functions are generic; without loss, we may assume that $\psi_1$ and $\psi_2$ are 
Morse functions, hence have finitely many critical points. 
We also assume that the set of critical points of $\psi_1$ is disjoint from that of $\psi_2$, which can be achieved by an arbitrary small translation $\psi_1(\cdot)\mapsto \psi_1(\cdot+a),\ a\in \R^2$ small. Thus we conclude for all $x\in \T^2$ either $D_x\psi_1$ or $D_x\psi_2$ is nonzero. 

Noting that $\det D_x \psi = \| D_x \psi_1 \wedge D_x \psi_2\|$, we introduce the function
\[
\Psi(\psi_1,\psi_2,z)= D_x \psi_1 \wedge D_x \psi_2 \, .
\]
We seek to apply the following consequence of the Sard-Smale theorem to the functional $\Psi$. 
 \begin{thm}[Theorem 5.4 of  \cite{H}]\label{ThmHutchings}
 Let $Y,Z$ be separable Banach manifolds, $E\to Y\times Z$ a Banach space fiber bundle and $\Psi:\ Y\times Z\to E$ a smooth section.  Suppose we have for all $(y,z)\in \Psi^{-1}(0)$
\begin{enumerate}
\item the differential $\nabla \Psi(y,z):\ T_{y}Y\times T_z Z\to E_{(y,z)}$ is surjective;
\item the partial derivative $\partial_z \Psi(y,z):\ T_z Z\to E_{(y,z)}$ is Fredholm of index $\ell$;
\end{enumerate}
Then for generic $y\in Y$, the set $\{z\in Z \ |\ \Psi(y,z)=0\}$ is an $\ell$-dimensional submanifold of $Z$. 
 \end{thm}
We apply Theorem \ref{ThmHutchings} with $Y=\big(C^k(\T^2,\R^2)\big)^2,\ Z=\T^2$, and $E = Y \times \Lambda^2(\T^2)$, where
$\Lambda^2(\T^2)$ is the vector bundle of differential 2-forms on $\T^2$.
 Concretely, we identify
$E \cong (C^k(\T^2, \R^2))^2 \times \T^2 \times \R$ 
using $T \T^2 \cong \T^2 \times \R^2$ and
 $\bigwedge^2(\R^2) \cong \R$. In particular, item 2 is always satisfied since 
 $\pd_x \Psi (\psi_1, \psi_2,x): T_x \T^2 \to E_{(\psi_1, \psi_2, x)}$ is a linear mapping between two finite-dimensional spaces. 

It remains to check item 1. The derivative of $\Psi$ acting on $(h_1, h_2, v) \in 
(C^k(\T^2, \R))^2 \times \R^2$ is given by
\begin{align*}
D_{(\psi_1, \psi_2, x)} \Psi (h_1, h_2, v) & = 
D_x h_1 \wedge D_x \psi_2 + D_x \psi_1 \wedge D_x h_2 + \\
& (D^2_x \psi_1 (v)) \wedge D_x \psi_2 + D_x \psi_1 \wedge (D^2_x \psi_2(v)) \, .
\end{align*}
It suffices to check surjectivity of $D_{(\psi_1, \psi_2, x)} \Psi$ at all $(\psi_1, \psi_2, x)$
such that $\psi_1, \psi_2$ are Morse. With $x$ fixed, by symmetry 
we can assume without loss that $D_x \psi_1 \neq 0$. Set $h_1 = 0, v = 0$, and construct
$h_2$ so that $D_x h_2$ is not parallel with $D_x \psi_1$. Then, 
$D_{(\psi_1, \psi_2, x)} \Psi(h_1, 0, 0) = D_x \psi_1 \wedge D_x h_2 \neq 0$, 
hence $D \Psi$ is surjective at $(\psi_1, \psi_2, x)$. This completes the proof. 
\end{proof}

\subsection{Proof of Theorem \ref{thm:strongCoupling}}

\begin{proof}
For ease of notation and to avoid factors of $2 \pi$, 
we work below with the parameterization $\T^2 \cong [0,2\pi)^2$. 
By Lemma \ref{lem:F2Generic}, it suffices to check the transversality condition \ref{eq:psiTransCond} for the function
\[
\psi(x_1,x_2) = \begin{pmatrix}
\sin x_1 +  \sin  (x_2 - x_1) \\ 
\sin x_2 + \sin  (x_1 - x_2) 
\end{pmatrix} \, . 
\]
That is, we seek to show that 
the system $\det D_x \psi = 0 , D_x \det D \psi = 0$ does not have any solutions. 
This system of equations is given by
\begin{align}
\cos(x_1)\cos(x_2)+(\cos(x_1)+\cos(x_2))\cos(x_1-x_2)=0 \label{eq1}\\
-\sin(x_1)\cos(x_2)-\sin x_1\cos(x_1-x_2)-(\cos x_1+\cos x_2)\sin(x_1-x_2)=0 \label{eq2}\\
-\sin(x_2)\cos(x_1)-\sin x_2\cos(x_1-x_2)+(\cos x_1+\cos x_2)\sin(x_1-x_2)=0 \label{eq3} 
\end{align}
Adding \eqref{eq2} to \eqref{eq3} gives
\begin{align}\label{eq4}
\sin(x_1)\cos(x_2)+\sin x_1\cos(x_1-x_2)+\sin(x_2)\cos(x_1)+\sin x_2\cos(x_1-x_2)=0 \, .
\end{align}
 Solving \eqref{eq4} and \eqref{eq1} for $\cos(x_1 - x_2)$ separately yields
\[
\frac{1}{\sec (x_1)+\sec (x_2)}=\frac{\sin (x_1+x_2)}{\sin
   (x_1)+\sin (x_2)} \, ;
\]
 division by $\cos x_1, \cos x_2$ is justified since, as one can easily check, no solution
$(x_1, x_2)$ can satisfy either of $\cos x_1 = 0, \cos x_2 = 0$. With some standard algebraic manipulations, this equation can be cast as
\[
(\sin x_1+\sin x_2)(1-\sin x_1\sin x_2)=0 \, .
\]

If $\sin x_1 \sin x_2 = 1$, then $\sin x_1=\sin x_2=\pm1$, 
which is inconsistent with \eqref{eq2}. If
$\sin x_1+\sin x_2=0$, then $x_1+x_2=2k\pi$ or $x_1-x_2=(2k+1)\pi$ for some $k \in \Z$. If
 $x_1+x_2=2k\pi$, \eqref{eq1} will give us $\cos(2x_1)=-\frac{\cos x_1}{2}.$ Plugging this into \eqref{eq3}, we get $\sin(2x_1)=\frac{\sin x_1}{4}.$
Since $\frac{\cos^2x_1}{4}+\frac{\sin^2 x_1}{16}=1$ is a contradiction, we 
deduce that no solution to \eqref{eq1},\eqref{eq2},\eqref{eq3} can satisfy $x_1 + x_2 = 2 k \pi$. 
Similarly, one can rule out solutions satisfying $x_1-x_2=(2k+1)\pi$. This completes the proof. 
\end{proof}

\section{{Construction of noise models satisfying conditions (E), (C), (ND)}}\label{sec:noiseModel}

\newcommand{\Sk}{\operatorname{Skew}}

In this section we will construct an explicit example of a random, volume-preserving diffeomorphism $\omega \mapsto R_\omega$ satisfying assumptions (E), (C), (ND). 
Throughout, for ease of notation we write $d = 2 N$. We write $\Oc(d)$ for the space
of orthogonal $d \times d$ matrices, and $\rho$ for the geodesic distance on $\Oc(d)$. 
Let $\Sk(d) = T_{\Id} \Oc(d)$ denote the Lie algebra of skew-symmetric $d \times d$ matrices.

\subsection*{Preliminary construction}
Recall that $\T^d$ is parametrized by $[0,1)^d$. 
Let $\{ z_i\}_{i = 1}^K$ be a collection of points with the property that
\begin{align}\label{cover120}
\T^d = \cup_i B_{1/20}(z_i) \, , 
\end{align}
i.e., the balls $\{ B_{1/20}(z_i)\}$ of radius $1/20$ cover $\T^d$.

Let $\psi : [0,\infty) \to [0,1]$ be a $C^\infty$ bump function such that $\psi |_{[0,1/10]} \equiv 1$
and $\psi|_{[1/5,\infty)} \equiv 0$. For $z \in \T^d$, define $\Delta_i(z) \in [-1/2, 1/2)^d \subset \R^d$
to be the unique vector such that $z - z_i = \Delta_i(z)$ modulo 1. 

For $U \in \Sk(d)$, define
\[
\Phi_U^{(i)} : \T^d \to \T^d \, , \quad \Phi_U^{(i)}(z) 
= z_i + \exp( \psi(d(z, z_i)) U) \Delta_i(z)
\]
where in the above formula we regard all coordinates in the RHS modulo 1. 
This yields a defined, continuous mapping of $\T^d$ into itself. 

Geometrically, along shells of constant distance from $z_i$, $\Phi^{(i)}_U(z)$ is a rigid rotation; as distance from $z_i$ increases, the rotation diminishes to 0. With this picture in mind, it is intuitively clear that $\Phi^{(i)}_U$ preserves volume on $\T^d$. Below, we check this with linear algebra.

\begin{cla}
The mapping $z \mapsto \Phi^{(i)}_U(z)$ is a volume-preserving diffeomorphism. 
\end{cla}
\begin{proof}
For simplicity, we take $z_i = 0$ and
will check these properties for $\Phi_U(z) = \exp(\psi(|z|) U) z$ defined for $z \in \R^d$. 
For $v \in T_z \R^d \cong \R^d$, we have
\[
D_z \Phi_U (v) = \exp (\psi(|z - z_i|) U) \left( \langle \hat z, v \rangle \psi'(|z - z_i|) U z
+ v \right)\, , 
\]
where $\hat z = z / |z|$ and here we use that for $t \in \R$, $\exp(t U) U = U \exp(t U)$ (note that there is no differentiability issue at $z = 0$). Since 
$\exp(U) \in \Oc(d)$, it suffices to check
that 
\[
v \mapsto A v := \langle \hat z, v \rangle \psi'(|z|) U z + v 
= (|z| \psi'(|z|) U \hat z \otimes \hat z^\top + \Id) v
\]
 has determinant 1. 

Define $\Pi v = \langle v, \hat z \rangle \hat z$ and $\Pi^\perp = \Id - \Pi$. 
Let $v \in \R^d$ and decompose $v = \Pi v + \Pi^\perp v =: \a \hat z + v^\perp$. 
Then,
\begin{gather*}
\Pi (A v) =  |z| \psi'(|z|) \langle \hat z, v \rangle \Pi U \hat z + \Pi v = \a \hat z \quad \text{ and } \\
\Pi^\perp (A v^\perp) = |z| \psi'(|z|) \langle \hat z, v \rangle \Pi^\perp U \hat z + 
\Pi^\perp v = \a |z| \psi'(|z|)  U \hat z + v^\perp \,. 
\end{gather*}
having used that $\Pi U \hat z = 0, \Pi^\perp U \hat z = U \hat z$ (recall that $\langle w, U w \rangle \equiv 0$ when $U \in \Sk(d)$). 
Expressing $A$ in block-matrix form, we have
\[
A v = \begin{pmatrix} \Pi A v \\ \Pi^\perp A v 
\end{pmatrix}
= \begin{pmatrix}
\Id & 0 \\
 |z| \psi'(|z|) U \hat z \otimes \hat z^\top & \Id
\end{pmatrix} 
\begin{pmatrix}
\Pi v \\ \Pi^\perp v
\end{pmatrix} \, .
\]
This clearly has determinant 1. 
\end{proof}

\subsection*{Construction of $\Omega_0, R_\omega$}

Given $h \in \R^d$, define $T_h : \T^d \to \T^d$ to be the translation $T_h z = z + h$, again
regarding all coordinates on the RHS modulo 1. Given $U^{(1)}, \cdots, U^{(K)} \in \Sk(d)$ and $v \in \R^d$, we define
\[
R_{(v; \{U^{(i)}\})} := T_v \circ \Phi^{(K)}_{U^{(K)}} \circ \cdots \circ \Phi^{(1)}_{U^{(1)}} \, .
\]
With $\Omega_0 = 
\R^d \times \Sk(d)^K$, we see that $\Omega_0 \ni \omega \mapsto R_\omega$
yields a `random' volume-preserving diffeomorphism of $\T^d$.

Below, we regard $\Omega_0$ as a copy of $\R^{d + K d (d - 1)/2}$ equipped with
Lebesgue measure $\Lambda$ and the standard Euclidean norm.
\begin{prop} 
There exists $c = c_{K,d} > 0$ sufficiently small so that the following holds.  
Let $\P_0$ be any Borel probability measure on $\Omega_0$ such that
\begin{itemize}
\item[(i)] $\operatorname{Supp}( \P_0)$ is contained in the ball of radius $c_K$ centered at the origin; 
\item[(ii)] $\P_0 \ll \Lambda$ with $\| d \P_0 / d \Lambda\|_{L^\infty} < \infty$; and
\item[(iii)] $\exists \zeta > 0$ such that $d \P_0 / d \Lambda > 0$ on the ball of radius $\zeta$
centered at the origin. 
\end{itemize}
Then, $R_\omega$ equipped with $\P_0$ satisfies conditions (E), (C) and (ND). 

\end{prop}
\begin{proof}
Throughout, for $(z, E) \in \Gr_{d/2}(\T^d)$ fixed, we write
\[
\Psi_{(z, E)} : \Omega_0 \to \Gr_{d/2}(\T^d) \, , \quad 
\Psi_{(z, E)}(\omega) := (R_{\omega}z , D_z R_{ \omega}(E)) \, .
\]
Here and throughout, elements $\omega \in \Omega_0$ are written $\omega = (v, (U^{(i)}))$. 

Conditions (C) and (E) are straightforward, and follow 
from the fact that for fixed $(z, E) \in \Gr_{d/2} (\T^d)$, 
the mapping $\Psi = \Psi_{(z, E)}$ is a continuous mapping sending the origin to $(z, E)$. 
Condition (C) uses item (i), while (E) uses (ii) and (iii). 

It remains to check condition (ND). For this, by a compactness argument and the Constant Rank Theorem, it suffices to show that for $(z, E) \in \Gr_{d/2}(\T^d)$ fixed, the mapping
$\Psi = \Psi_{(z, E)}$ is a submersion. To simplify the argument, we begin with the following
observation regarding the `upper triangular' structure of $D \Psi$: 
writing $v = (v_1, \cdots, v_d) \in \R^d$ and $z = (z_1, \cdots, z_d)$\footnote{That is, we are abandoning for the moment the distinction between $x$ and $y$ coordinates in $\T^d = \T^{2N}$},
we have that 
\[
D_{(t, (U^{(i)}))} \Psi \left(\frac{\pd}{\pd v_i} \right) = \frac{\pd}{\pd z_i} \, .
\]
That is, varying $v$ does not change at all the $\Gr_{d/2}(\R^d)$ coordinate in the image. 
Therefore, it suffices to show that when $v$ is held fixed, we have that
\[
(U^{(i)}) \mapsto  D_z R_{(v, (U^{(i)}))}(E)
\]
is a submersion $\Sk(d)^K \to \Gr_{d/2}(\R^d)$. 

In fact, we will show that for each $z$, it suffices to consider tangent directions
corresponding to a single $U^{(i)}$. To see this, we make the following claim.

\begin{cla}\label{cla:findJ}
There exists $c = c_{K, d} > 0$ with the following property. 
For any $z \in \T^d$ there exists $1 \leq j \leq K$ such that for any $(U^{(i)}) \in \Sk(d)^K$, 
$\| (U^{(i)})\| \leq c$, we have that
\[
d(\Phi^{(j-1)}_{U^{(j-1)}} \circ \cdots \circ \Phi^{(1)}_{U^{(1)}} z, z_j) \leq \frac{1}{10} \, .
\]
\end{cla}
Indeed, the claim holds with $j$ any index for which $d(z, z_j) \leq 1/20$ (see \eqref{cover120}), assuming $c = c_{K, d}$
is taken small enough. 

With $(z, E)$ and the above value of $j$ fixed, we now set about checking that
\[
U^{(j)} \mapsto D_z R_{(t, U^{(i)})} (E)
\]
is a submersion. Since $T_v, \Phi^{(i)}_{U^{(i)}}$ are all diffeomorphisms of $\T^d$, 
it suffices to check that $U^{(j)} \mapsto D_{z'} \Phi^{(j)}_{U^{(j)}}(E')$ is a submersion
$\Sk(d) \to \Gr_{d/2}(\T^d)$, where $z' = \Phi^{(j-1)}_{U^{(j-1)}} \circ \Phi^{(1)}_{U^{(1)}}(z)$
and $E' = D_z \Phi^{(j-1)}_{U^{(j-1)}} \circ \Phi^{(1)}_{U^{(1)}}(E)$. 
By our choice of $j$, Claim \ref{cla:findJ} ensures $d(z', z_j) \leq 1/10$, hence
$D_{z'} \Phi^{(j)}_{U^{(j)}} = \exp(U^{(j)})$. In view of the composition
\[
U^{(j)} \mapsto \exp(U^{(j)}) \mapsto \exp(U^{(j)})(E') 
\]
and the fact that $U \mapsto \exp(U)$ is a local diffeomorphism $\Sk(d) \to \Oc(d)$, it suffices
to check that $O \mapsto O(E)$ is a submersion $\Oc(d) \mapsto \Gr_{d/2}(\R^d)$. 
Since surjectivity of a derivative is an open property, it suffices to check that the differential of
$O \mapsto O(E)$ is a submersion at the identity $\Id \in \Oc(d)$. 
\end{proof}
\begin{cla}
Fix $1 \leq k \leq d$ and $E_0 \in \Gr_k(\R^d)$. Define $\Xi : \Oc(d) \to \Gr_k(\R^d)$, 
$\Xi(O) := O(E_0)$. Then, $D_{\Id} \Xi : \Sk(d) \to T_{E_0} \Gr_k(\R^d)$ is surjective.
\end{cla}
\begin{proof}[Proof of claim]
We evaluate the differential explicitly in coordinates. 
Recall the chart $\Uc_{E_0} \cong L(E_0, E_0^\perp)$ for $\Gr_k(\R^d)$ at $E_0$. 
As one can check, in this chart, $\Xi(O) = O(E_0)$ is represented as
\[
\Xi(O) = \graph_{E_0} g(O) \, , \quad g(O) := (\Id - \Pi_{E_0}) O (\Pi_{E_0} O|_{E_0})^{-1} \, .
\]
Therefore, in these coordinates we have (writing $\Pi_{E_0} = \Pi, \Pi^\perp = \Id - \Pi_{E_0}$)
\[
D_O g (U)= \Pi^\perp U (\Pi O|_{E_0})^{-1} + 
\Pi^\perp O (\Pi O|_{E_0})^{-1} U (\Pi O|_{E_0})^{-1} 
\]
for $U \in T_O \Oc(d)$. Evaluating at $O = \Id$, we see that 
\[
D_{\Id} g(U) = \Pi^\perp U|_{E_0}
\]
This is clearly surjective as a linear mapping $\Sk(d) \mapsto L(E_0, E_0^\perp)$; given an arbitrary $B \in L(E_0, E_0^\perp)$, we have $D_{\Id} g(U) = B$ for any $U$ of the form
\[
U = \begin{pmatrix} \Pi U|_{E_0} & \Pi U|_{E_0^\perp} \\ 
				\Pi^\perp U|_{E_0} & \Pi^\perp U|_{E_0^\perp} \end{pmatrix}
				= 
				\begin{pmatrix} * & - B^\top \\ B & * \end{pmatrix} \, . \qedhere
\]
\end{proof}

\appendix

\section{A version of the Singular Value Decomposition}\label{sec:SVDapp}

Here we recall a version of the Singular Value Decomposition and related results 
used in this paper. 
Below, $d \geq 1$ and $A$ is a $d \times d$ matrix. The \emph{singular values} 
$\sigma_1(A) \geq \cdots \geq \sigma_d(A)$ 
are defined to be the eigenvalues of $A^\top A$, listed in decreasing order and counted
with multiplicity.

\begin{thm}[Singular Value Decomposition]\label{thm:svdApp}
There exist orthonormal bases $\{e_1, \cdots, e_d\}$ and 
$\{e_1', \cdots, e_d'\}$ of $\R^d$ with the property that
\[
A e_i = \sigma_i(A) e_i' \, .
\]
\end{thm}
The most important part of the proof of Theorem \ref{thm:svdApp} is to check that 
the eigenvalues of $A^\top A$ and $A A^\top$ coincide (multiplicities counted). 
Indeed, $\{ e_i\}$ is an (orthonormal) eigenbasis for $A^\top A$, while $\{ e_i'\}$ is an appropriate 
ordering of an (orthonormal) eigenbasis for $A A^\top$. 

The following min-max principles for singular values are also used in this paper: 
\begin{lem}
For all $1 \leq i \leq d$, the following hold. 
\begin{itemize} 
\item[(a)] 
\[
\sigma_i(A) = \max_{\substack{E \subset \R^d \\ \dim E = i}} \min_{\substack{v \in E \\ \| v \| = 1}}
\| A v \| \, .
\]
\item[(b)] 
\[
\prod_{j =1}^i \sigma_i(A) = \max_{\substack{E \subset \R^d \\ \dim E = i}} \det(A|_E) \, , 
\]
where $A|_E$ is regarded as a linear mapping $E \to A(E)$. 
\end{itemize}
\end{lem}

Lastly, we state the following corollary of Theorem \ref{thm:svdApp}, which we use in 
Section \ref{subsec:SVD} to estimate singular directions. 

\begin{lem}\label{coneSVD}
Let $\Cc \subset \R^d$ be a cone\footnote{For our purposes, a cone is defined to be a subset $\Cc$ of $\R^d$ with the property that $v \in \Cc \Rightarrow \lambda v \in \Cc$ for all $\lambda \in \R$.} and $k < d$. Assume $A$ is invertible, and has the property that for any $k$-dimensional subspace $E \subset \Cc$, we have that $A(E) \subset \Cc$ and $A^\top(E) \subset \Cc$. Then, 
$\exists 1 \leq i_1 < \cdots < i_k \leq d$ such that $e_{i_j}, e_{i_j}' \in \Cc$ for $1 \leq j \leq k$. 
\end{lem}

\bibliography{biblio}
\bibliographystyle{plain}

\section*{Acknowledgment}
Alex Blumenthal is supported by the National Science Foundation under Award No. DMS-1604805.
Jinxin Xue is supported by NSFC (Significant project No.11790273) in China and Beijing Natural Science Foundation (Z180003).

\Addresses

\end{document}